\documentclass[10pt,draft]{article}

\usepackage[textwidth = 430 pt, textheight = 630 pt ]{geometry}
\usepackage{amsmath,amsthm, amsfonts}
\usepackage{graphicx, rotating}

\usepackage[all,cmtip]{xy}

\newtheorem{thm}{THEOREM}
\newtheorem*{thm*}{THEOREM}
\newtheorem{cor}{Corollary}
\newtheorem{lem}{Lemma}
\newtheorem{prop}{Proposition}
\theoremstyle{definition}
\newtheorem{defn}{Definition}
\newtheorem{assump}{Assumption}
\theoremstyle{remark}

\newcommand{\bb}{\mathbf{b}}

\newcommand{\bF}{\mathbf{F}}

\newcommand{\bh}{\mathbf{h}}

\newcommand{\bK}{\mathbf{K}}

\newcommand{\bL}{\mathbf{L}}

\newcommand{\bR}{\mathbf{R}}

\newcommand{\bu}{\mathbf{u}}

\newcommand{\bv}{\mathbf{v}}

\newcommand{\bx}{\mathbf{x}}

\newcommand{\bg}{\mathbf{g}}

\newcommand{\beq}{\begin{equation}}
\newcommand{\eeq}{\end{equation}}
\newcommand{\bseq}{\begin{subequation}}
\newcommand{\eseq}{\end{subequation}}
\newcommand{\refeq}[1]{(\ref{#1})}

\renewcommand{\span}{\mathrm{span}\,}

\newcommand{\tr}{\mathrm{tr}\,}

\newcommand{\fg}{\mathfrak{g}}
\newcommand{\fh}{\mathfrak{h}}
\newcommand{\fk}{\mathfrak{k}}
\newcommand{\fn}{\mathfrak{n}}
\newcommand{\fa}{\mathfrak{a}}
\newcommand{\fp}{\mathfrak{p}}

\newcommand{\fs}{\mathfrak{s}}
\newcommand{\fu}{\mathfrak{u}}

\newcommand{\p}{\partial}

\newcommand{\cA}{\mathcal{A}}
\newcommand{\cB}{\mathcal{B}}

\newcommand{\cU}{\mathcal{U}}
\newcommand{\cR}{\mathcal{R}}
\newcommand{\cC}{\mathcal{C}}
\newcommand{\cD}{\mathcal{D}}
\newcommand{\cE}{{\mathcal E}}

\newcommand{\cG}{{\mathcal G}}

\newcommand{\cL}{{\mathcal L}}
\newcommand{\cM}{{\mathcal M}}
\newcommand{\cN}{{\mathcal N}}

\newcommand{\cS}{{\mathcal S}}

\newcommand{\cV}{{\mathcal V}}

\newcommand{\Cset}{{\mathbb C}}

\newcommand{\Rset}{{\mathbb R}}
\newcommand{\Sset}{{\mathbb S}}

\newcommand{\la}{\lambda}

\newcommand{\de}{\delta}

\newcommand{\al}{\alpha}

\newcommand{\ga}{\gamma}
\newcommand{\Ga}{\Gamma}
\newcommand{\ep}{\epsilon}

\newcommand{\om}{\omega}
\newcommand{\Om}{\Omega}
\newcommand{\si}{\sigma}

\newcommand{\half}{\frac{1}{2}}

\newcommand{\diag}{\mbox{diag}}

\newcommand{\bna}{\begin{eqnarray}}
\newcommand{\ena}{\end{eqnarray}}
\newcommand{\bea}{\begin{eqnarray*}}
\newcommand{\eea}{\end{eqnarray*}}
\newcommand{\ben}{\begin{enumerate}}
\newcommand{\een}{\end{enumerate}}
\newcommand{\bi}{\begin{itemize}}
\newcommand{\ei}{\end{itemize}}

\newcommand{\C}{{\mathcal C}}

\newcommand{\RR}{{\mathbb R}}

\newcommand{\Hset}{{\mathbb H}}
\newcommand{\CC}{{\mathbb C}}

\newcommand{\Res}{\mbox{Res}}

\begin{document}

\title{Integrability and Vesture for\\ Harmonic Maps into Symmetric Spaces}%
\author{Shabnam Beheshti\footnote{beheshti@math.rutgers.edu} $\,\,$and$\,$ Shadi Tahvildar-Zadeh\footnote{shadi@math.rutgers.edu}\\
Department of Mathematics\\
Rutgers, The State University of New Jersey\\
New Brunswick, NJ, USA 08854}%


\date{}%

\maketitle

\begin{abstract}
After giving the most general formulation to date of the notion of integrability for axially symmetric harmonic maps from $\RR^3$ into symmetric spaces,   we give a complete and rigorous proof that, subject to some mild restrictions on the target, all such maps are integrable.  Furthermore, we prove that a variant of  the inverse scattering method, called {\em vesture} (dressing) can always be used to generate new solutions for the harmonic map equations starting from any given solution.  In particular we show that the problem of finding $N$-solitonic harmonic maps into a noncompact Grassmann manifold $SU(p,q)/S(U(p) \times U(q))$ is completely reducible via the vesture (dressing) method to a  problem in linear algebra which we prove is  solvable in general. We illustrate this method by explicitly computing a 1-solitonic harmonic map for the two cases $(p=1,q=1)$ and $(p=2,q=1)$; and we show that the family of solutions obtained in each case contains respectively the {\em Kerr family} of solutions to the Einstein vacuum equations, and the {\em Kerr-Newman} family of solutions to the Einstein-Maxwell equations.
\end{abstract}

\section{Introduction}
Unlike Maxwell's linear electromagnetic field theory, which is generally solvable by linear superposition of plane-wave solutions, the explicit solvability of geometric field theories such as  principal chiral field models, non-linear sigma models, Yang-Mills connections and Einstein equations, is seriously hampered by their highly nonlinear character, unless the number of independent
variables can be reduced to 2 via suitable symmetry assumptions. In that case all of the above mentioned geometric field theories are {\em completely integrable} \cite{ZB-I, ZB-II, Ale80, GurXan82, NeuKra83, Maz84, EGK84}. Recall that a nonlinear system of PDEs is said to be completely integrable if it can be cast as the compatibility condition(s) of an overdetermined linear system of differential equations, called a ``Lax Pair" or Lax System (see \cite{Lax68, AKNS-1} and references therein). 

Interestingly, all of the geometric field theories named above, in their complete integrability regime, are instances of {\em harmonic maps}. 
Such maps generalize the notion of geodesics to higher-dimensional domains, and are the simplest of all nonlinear geometric field theories.
 The Principal Chiral Field model involves harmonic maps into Lie groups. Its integrability  has been addressed in \cite{ZM, ZS-II, Uhl89}.  Nonlinear Sigma-Models are nothing but harmonic maps into symmetric spaces. Their integrability has been studied in \cite{Poh76, ZM, EicFor80, NSanch82, Woo94}. More generally, harmonic maps from surfaces into compact symmetric spaces and their connections to integrability  have also been explored \cite{TerUhl04, Terng10}. On the other hand,
stationary axisymmetric Einstein Vacuum and Einstein-Maxwell Equations reduce to harmonic maps into the real and complex hyperbolic plane, respectively \cite{Maz84, Wei92, Wei96}, and thus the study of their integrability \cite{ZB-I, ZB-II, Ale80, EGK84} can also be subsumed into that of harmonic maps. Here we recall that by ``stationarity'' and ``axisymmetry'' of a solution to the Einstein equations of general relativity and gravitation one means the existence of two {\em Killing fields}, one timelike generating an $\RR$ action and the other spacelike generating a circle action,  for the metric  tensor of a four-dimensional Lorentzian manifold, which is the unknown in these equations\footnote{If both of the two Killing fields are spacelike, then the equations reduce to a {\em wave map}, i.e. the hyperbolic analogue of a harmonic map.  One could also assume other types of actions for the Killing fields, e.g. both of them generating circle actions, etc.  In the context of globally hyperbolic and asymptotically flat spacetimes, the most natural choice is the one we are considering here.  Our results however generalize without difficulty to wave maps and other cosmological  situations.}.

In this paper we show that the integrability results for the above-named models are special instances of a more general theorem, which we prove, namely that {\em axisymmetric harmonic maps of $\RR^3$ into symmetric spaces are completely integrable}. 

A key feature of completely integrable systems is the possibility of implementing the Inverse Scattering Mechanism (ISM) to find  new exact solutions from old ones.
This has been exploited in particular for scalar equations, such as the Korteweg-de Vries (KdV) hierarchy \cite{GGKM, AKNS-2}, and the cubic nonlinear Schr\"odinger equation \cite{ZakSha71, ZS-II}, for which the inverse scattering transform is computable by solving
the Gelfand-Levitan-Marchenko integral equations.
These integral equations are however less susceptible to explicit treatment when dealing with second-order evolution equations such as sine-Gordon \cite{AKNS-3,  ZS-I, SG, ShatStrauss96}, or
non-scalar-valued systems such as harmonic maps. In the case of sine-Gordon, alternatives to ISM such as the B\"acklund transform, have proved their utility as solution-generating mechanisms.
{\em Here we will develop a workable approach for obtaining new harmonic maps from old ones by supplying the rigorous foundation for  a method akin to the B\"acklund transform, called the {\em vesture} or {\em dressing} technique}, explained below. 
 
Expressed informally, we summarize the main results of this paper as follows:
\begin{thm*}
Let $G$ be a real semisimple Lie group and let $K$ be a maximal compact subgroup of $G$.  Then any axially symmetric harmonic map from $\mathbb{R}^3$ into the Riemannian symmetric space $G/K$ satisfies an integrable system of equations.  Furthermore, it is always possible to generate new harmonic maps from a given one using the vesture (dressing) method.
\end{thm*}
A precise statement appears in Theorem  \ref{eq:mainthm}.

  The results of our paper go beyond the current literature by providing a general framework for the study of harmonic maps into the \emph{noncompact} symmetric spaces commonly appearing in mathematical physics.  We also provide mathematically rigorous proofs, both of the integrability claim and also of the solvability of the algebraic equations obtained through the dressing technique, under quite general assumptions. To our knowledge, many of the issues we have encountered and overcome in providing these rigorous proofs have not been previously addressed in the literature.  

We furthermore demonstrate how the dressing technique can be employed to construct {\em by purely algebraic means}, a $2nN$-parameter family of harmonic maps, for any integer $N$, into any {\em noncompact Grassmann manifold} $SU(p,q)/S(U(p) \times U(q))$ with $p+q = n$, starting from any given harmonic map.  As an explicit example, this task is then  carried out for $N=1$ in the two cases $(p=1,q=1)$ and $(p=2,q=1)$.  In both of these cases the initial solution, the so-called  ``seed", when viewed as the metric of a stationary axisymmetric spacetime, corresponds to the Minkowski metric.  We show that in the first case,  the four-parameter family of the so-called ``1-solitonic" maps one obtains, when viewed as a solution of Einstein's Vacuum Equations, contains the {\em Kerr family} of spacetime metrics \cite{Ker63}, while in the second case the six-parameter family obtained contains the {\em Kerr-Newman family} \cite{NCCEPT65}.  We thus provide a complete, rigorous, and at the same time concise derivation of two of the most significant exact solutions of Einstein's Field Equations. These examples also suggest that the approach via dressing may  make the task of generating meaningful solutions for other effectively two-dimensional geometric field theories much more tractable.

We now briefly summarize the main points of inverse scattering and vesture methods in the above context.
The classical ISM can be described heuristically as follows:  given a nonlinear {\em first-order} evolution equation 
 \begin{equation} \label{eq:unonlin} u_t = F(u, u_x, u_{xx}, \ldots),\end{equation}
for a scalar function $u=u(t,x)$, one considers an associated eigenvalue problem $L\psi = \lambda \psi$ for an {\em isospectral} family of linear differential operators, i.e.  $L=L(t)$ such that $\lambda_t =0$.  In the KdV case, for example, $L=-\frac{\partial^2}{\partial x^2} + u(t,x)$.  The {\em direct scattering problem} consists of finding a {\em scattering matrix} $S(t, \lambda)$ with the property that (loosely speaking) $\displaystyle{ \lim_{x\rightarrow \infty} \psi(x,t,\lambda) = S(t,\lambda)\cdot \lim_{x \to -\infty}\psi(x,t,\lambda)}$.  Note that $S$ is a matrix because the asymptotic eigenspaces are multi-dimensional.  Given Cauchy data $u(0,x)$ for (\ref{eq:unonlin}) one may use direct scattering to find $S(0,\lambda)$.  Now, it turns out that the evolution in $t$ of $S(t,\lambda)$ is governed by a linear equation, and also that the eigenfunctions $\psi$ satisfy a secondary equation $\frac{\partial}{\partial t}\psi = B \psi$, where $B= B(u, u_x, u_{xx}, \ldots)$ is another differential operator.  The isospectral condition implies that $B$ satisfies a compatibility condition with $L$, namely $L_t = [B,L]$, which agrees precisely with the nonlinear evolution equation (\ref{eq:unonlin}) of interest.  For such cases, by solving the \emph{inverse scattering problem}, i.e. the Gelfand-Levitan-Marchenko (GLM) integral equations, one recovers the potential $u(t,x)$ from a given $S(t, \lambda)$, thus solving the original evolution equation (see Figure~\ref{classical}).  PDEs to which this procedure applies, e.g. KdV, are referred to as being \emph{integrable by way of ISM}.
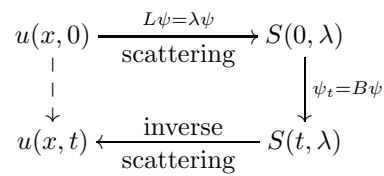
\begin{figure}[h]
\[
\xymatrixcolsep{5pc}
\xymatrix{
u(x,0)	\ar@{-->}[d] \ar[r]^{L\psi=\lambda \psi}_{\mbox{scattering}}	& S(0,\lambda) \ar[d]^{\psi_t = B \psi} \\
u(x,t)								& S(t, \lambda) \ar[l]^{\mbox{scattering}}_{\mbox{ inverse}}
}
\]
\caption{Classical ISM}\label{classical}
\end{figure}

This technique may also be utilized to address certain {\em second-order} evolution equations, such as the sine-Gordon equation.  In that case the PDE  appears as the {\em compatibility condition} $$U_t-V_x +[U,V]=0$$ for a system of \emph{matrix equations} $$\Psi_x = U \Psi,\qquad\Psi_t= V\Psi,$$ for a matrix function $\Psi = \Psi(x,t,\lambda)$, where $U,V$ are given $2\times 2$ matrices depending on the solution $u(t,x)$ of the sine-Gordon equation, as well as on the so-called {\em spectral parameter} $\lambda$.  
The B\"acklund transform, for example, can now be used in place of solving the GLM integral equations, to obtain new solutions for the sine-Gordon equation.

 A solution-generating mechanism closely related to the above procedure for sine-Gordon was introduced in \cite{ZB-I} to treat Einstein's vacuum equations:  Assuming existence of two commuting Killing fields generating a continuous group $G$ of isometries for a Lorentzian manifold $(\mathcal{M},\bg)$, the vacuum equations $$\mathbf{R}_{\mu\nu}=0,\qquad \mu , \nu = 0, \ldots , 3$$
where $\mathbf{R}_{\mu\nu}$ is the Ricci curvature tensor of $\bg$, can be viewed as the compatibility conditions for a linear evolution problem
\begin{equation}\label{psisystem}\mathcal{L} \Psi = \Lambda \Psi,\qquad \left. \Psi \right|_{ \lambda = 0} = \bg'
\end{equation}
where $\bg'=\bg'(\mathbf{x})$ is the metric of the 2-dimensional quotient manifold $\mathcal{M}/G$, $\mathcal{L}$ is  a matrix operator involving  differentiation in the complex parameter $\lambda$ as well as in $\mathbf{x}$, and $\Lambda$ is a matrix depending on $\lambda$ and on $\bg'$.  In \cite{ZB-I}, the above system \refeq{psisystem} was shown to be  the Lax system for the reduced Einstein's Vacuum equations, which established their complete integrability.

As mentioned before, complete integrability alone does not suffice to show that there is a workable solution-generating method, because of inherent difficulties in solving the GLM equations.  In \cite{ZB-I} the authors provide an alternative:  One first chooses an ``initial seed" metric $\bg_0$ and solves (by any means possible) the  linear system \refeq{psisystem}  to obtain a generating matrix $\Psi_0$.  The evolution problem for the scattering matrix $S$  in the classical approach  is now replaced by a \emph{dressing} or \emph{vesture} technique, in which new solutions $\bg'$ are constructed by first solving a linear system of algebraic equations for  the \emph{dressing matrix} $\chi$, which has the property that $\Psi = \chi \Psi_0$ is a solution of (\ref{psisystem}), and then setting the parameter $\lambda$ equal to zero in order to recover $\bg'$ (see Figure~\ref{vacuum}); the metric $\bg$ on $\mathcal{M}$ solving the original vacuum equations is then recoverable from $\bg'$  by quadratures.  This procedure for generating $\bg'$ will be generalized and fully explained in Section 3.
\begin{figure}[h]
\[
\xymatrixcolsep{5pc}
\xymatrix{
\bg_0	\ar@{-->}[d] \ar[r]^{\mathcal{L}\Psi=\Lambda \Psi \qquad}_{\Psi|_{\lambda=0}=\bg_0'\qquad}	& \Psi_0(\bg_0',\lambda) \ar[d]^{\Psi = \chi \Psi_0} \\
\bg								& \Psi(\bg',\lambda) \ar[l]^{\lambda = 0,\ \mbox{quad.}}_{}	
}
\]
\caption{Vesture for Einstein's Equations}\label{vacuum}
\end{figure}

Incidentally, it was known already~\cite{Ern68a, Ern68b} (even though not cited in \cite{ZB-I}) that the Einstein vacuum equations in the stationary, axisymmetric case reduce to a single equation in terms of a complex-valued scalar function, called the {\em Ernst potential}.  It turns out that this is the equation for an axisymmetric {\em harmonic map} from $\mathbb{R}^3$ into the {\em hyperbolic plane}, which is the simplest example of a non-compact Grassmann manifold: $\mathbb{H}_\mathbb{R}  \cong SL(2, \mathbb{R})/SO(2) \cong SU(1,1)/S(U(1)\times U(1))$. Our thesis is that {\em the integrability of reduced Einstein equations is simply a special case of a more general phenomenon, namely the integrability of axisymmetric harmonic maps from $\RR^3$ into symmetric spaces}, which is the focus of our paper.

The rest of this paper is organized as follows:  In Section 2 we cover the preliminary background needed for the study of harmonic maps into Riemannian symmetric spaces.  Section 3 is devoted to establishing integrability of such maps, and showing how the vesture method is implemented for them.  In particular we provide the first rigorous proof that the resulting linear algebraic system is solvable in general. In Section 4 we specialize to the case of noncompact Grassmann manifolds $\cG_{p,q}$ and show how the problem of finding N-solitonic maps into them is reduced to solving a $2N\times 2N$ linear system.  We then carry out the computation explicitly for the case when the target of the map is either $\cG_{1,1}$ or $\cG_{2,1}$ and $N=1$.  In each case we choose a starting solution that corresponds to the Minkowski metric, and show how to obtain the Kerr, respectively Kerr-Newman solution in this way.    Avenues of further exploration are briefly discussed at the end of the paper.

\section{Harmonic Maps into Lie Groups and Symmetric Spaces}
In this section, we establish our notation and introduce the necessary terminology for defining harmonic maps into Lie groups and symmetric spaces.

\subsection{Lagrangian field theory}
We adopt the approach of  \cite{Christodoulou} for the general set-up: Let $(\cM^m,\bg)$ and $(\cN^n,\bh)$ be two Riemannian or pseudo-Riemannian manifolds.  Any differentiable mapping $f :\cM \to \cN$ can be viewed as a section of the {\em velocity bundle} $\displaystyle{\cV =\bigcup_{x\in\cM,q\in \cN} \bL(T_x \cM,T_q\cN)}$, where $\bL(V,W)$ denotes the space of linear transformations from vector space $V$ to vector space $W$.  Using local coordinates $(x,q,v_\mu^a)$ on $\cV$, such a section is given by $s_f(x) = (x,f(x),Df(x))$.  A {\em Lagrangian} $L$ is an $m$-form defined on $\cV$, i.e. $L = \ell(x,q,v_\mu^a)\ep[\bg]$ where $\ep[\bg]$ is the volume form of $(\cM,\bg)$.  Once evaluated on a section $s_f$, the Lagrangian becomes an $m$-form on $\cM$, so that it can be integrated on a domain $\cD$ in $\cM$, and the resulting functional is called the {\em action} corresponding to $L$: 
\beq\label{eq:action}
\cA[f,\cD] = \int_\cD L\circ s_f.
\eeq
A critical point of the action $\cA$, with respect to variations that are compactly supported in $\cD$, is a solution to the {\em Euler-Lagrange equations} for $L$ in $\cD$.
By analogy with classical Hamiltonian mechanics, the quantities $(v_\mu^a)$ are called the {\em canonical velocities}, and their duals with respect to the Lagrangian density $\ell$ are called the {\em canonical momenta} $$p_a^\mu := \frac{\p\ell}{\p v_\mu^a}.$$ The {\em canonical stress} is by definition the Legendre transform (with respect to the velocities) of the Lagrangian density:
$$T_\mu^\nu = v_\mu^a p^\nu_a - \de_\mu^\nu \ell.$$

Let $Y = (Y^\mu)$ be a vectorfield on $\cM$ and let $Z = (Z^a)$ be a vectorfield on $\cN$.  The {\em Noether current} corresponding to $(Y,Z)$ is $$j_{(Y,Z)}^\mu := p^\mu_a Z^a + T^\mu_\nu Y^\nu.$$
Let $J = *j$ be the Hodge dual of $j$ with respect to the metric $\bg$.  {\em Noether's Theorem} \cite{Noe18} states that, for any solution $f$ of the Euler-Lagrange equations, $$d(J\circ s_f) = K\circ s_f,\qquad K := \cL_Z L - \cL_Y L,
$$ where $\cL$ denotes the Lie derivative operator.  In particular, if $L$ is invariant under the Lie flow generated by $(-Y,Z)$ on the velocity bundle $\cV$, then $J\circ s_f$ is a closed $(m-1)$-form, so that its integral on any closed $m-1$-dimensional submanifold $\cS$ of $\cM$ is a {\em homological invariant} i.e. depends only on the homology class of $\cS$.

As particular examples, consider the case where $Y$ is a Killing field of $(\cM,\bg)$, i.e. $\cL_Y \bg = 0$.  If the Lagrangian density $\ell$ is invariant under the flow of $Y$ it then follows that $\cL_Y L = 0$, so that the corresponding Noether current $j$ having $Z=0$ is divergence-free: $\p_\mu j_{(Y,0)}^\mu = 0$.  Similarly, if $Z$ is a Killing field for $(\cN,\bh)$ and $\ell$ is invariant under the flow of $Z$, then once again one gets a divergence free current, namely $\p_\mu j_{(0,Z)}^\mu = 0$.

\subsection{Harmonic maps}
\begin{defn}
A {\em harmonic map} $f$ is a critical point, with respect to compactly supported variations, of the action $\cA$, where $L = \ell \ \ep[\bg]$ is the following Lagrangian
$$ \ell(x,q,v_\mu^a) := \half \bg^{\mu\nu}(x) \bh_{ab}(q) v_\mu^a v_\nu^b.$$
Therefore $L\circ s_f = \half \tr_\bg f^*\bh$.
\end{defn}

The harmonic map action is clearly invariant under the isometries of the domain $\cM$ and the target $\cN $.  Thus any Killing field of either of these manifolds will yield a conservation law for the harmonic map.  Consider in particular a Killing field $Z$ for the target $\cN$.  The corresponding Noether current, when evaluated on a solution section $s_f$, is
$$j^\mu = p^\mu_a Z^a = \bg^{\mu\nu}\p_\nu f^b\bh_{ab} Z^a = \bg^{\mu\nu} \phi_\nu = (\phi^\sharp)^\mu$$ where $\phi$ is the pull-back under $f$ of the 1-form $\zeta$, namely $\phi = f^*\zeta$, and $\zeta = Z^\flat$ is the dual of the vectorfield $Z$ with respect to the metric $\bh$, namely  $\zeta_b = \bh_{ab}Z^a$.


\subsection{Lie Groups}\label{liegroups}
Let $G$ be a  Lie group and $U$ an open domain in $\RR^n$, $n=\dim G$.  Given a (suitably regular) parametrization $g: U \to G$, the Lie-algebra-valued connection 1-forms $w = g^{-1} dg$ and $w' = -dg g^{-1}$ are called the {\em Maurer-Cartan} left- and right-invariant forms for $G$.  A left-invariant form $w$ gives rise to a left-invariant metric on $G$ in the following way:  Let $\{X_a\}_{a=1}^n$ be a basis for the Lie algebra $\fg$ of the Lie group $G$.  Thus $w = \zeta^a X_a$ where $\zeta^a \in \bigwedge^1(U)$ are 1-forms, with $\zeta^a = \zeta^a_\mu dx^\mu$, for the local coordinates $(x^\mu)=\bx\in U$. One computes
\beq\label{eq:hmunu}
\half \tr (w^2) = \half \zeta^a_\mu\zeta^b_\nu \tr (X_a X_b) dx^\mu dx^\nu = \eta_{ab} \zeta^a_\mu \zeta^b_\nu dx^\mu dx^\nu =: \bh_{\mu\nu} dx^\mu dx^\nu.
\eeq
Here $\eta$ is the Killing-Cartan quadratic form on $\fg$:  $\eta_{ab} := \half \tr(X_aX_b)= \frac{1}{6}C^c_{ad}C^d_{bc}$, where $C_{ab}^c$ are the {\em structure constants} of the Lie algebra, defined by $[X_a, X_b]=C^c_{ab}X_c$.  Note that $\eta$ is non-degenerate precisely when $\fg$ is {\em semisimple};  in this case, one sees that the tensor $\bh$ defined above provides a non-degenerate quadratic form on the tangent space $T_gG$ for any $g\in G$, and thus turns $G$ into a pseudo-Riemannian manifold, of signature $((\dim \fp)+,(\dim\fk)-)$, where $\fk,\fp$ are as in the \emph{Cartan decomposition} of $\fg$ (see below).

\subsection{Harmonic maps into Lie groups}
Consider now the case where the target manifold $(\cN,\bh)$ of a harmonic map $f: \cM \to \cN$ is a Lie group, and $\bh$ is the invariant metric defined in (\ref{eq:hmunu}).   Let $W := f^*w$ denote the pullback of the Maurer-Cartan form $w$ under $f$.  Thus $W = \phi^I X_I$ where $\phi^I = f^*\zeta^I$ as before, and $I$ is a counting index (not a component index).  Since both $d$ and $\wedge$ are covariant under pullbacks, so is the equation $dw+ w\wedge w = 0$, thus $dW + W\wedge W = 0$ where $d$ now denotes exterior differentiation on the domain $\cM$.  On the other hand, each dual vectorfield $Z_I =( \zeta^I)^\sharp$ is easily seen to be a Killing field for the metric $\bh$, and by Noether's Theorem gives rise to a divergence-free current $j_I = (\phi^I)^\sharp$.  Thus $W^\sharp = (\phi^I)^\sharp X_I$ is also divergence free.  The system of equations for a harmonic map can in this way be recast into the following (nonlinear) Hodge system for a Lie-algebra-valued connection 1-form $W\in \bigwedge^1(\cM,\fg)$:
\beq\label{hodgesys} dW + W\wedge W = 0,\qquad \delta W = 0,\eeq
where $\delta = *d*$ is the divergence operator, and $*$ denotes the Hodge dual with respect to the domain metric $\bg$.  It is the above Hodge system that, in situations where the domain is effectively two-dimensional, becomes the starting point of the quest for a Lax pair, which in turn allows the inverse scattering method to be applied.

\subsection{Symmetric Spaces}\label{symmsp}
Let $H$ be a  complex semisimple Lie group having Lie algebra $\fh$.  A \emph{real form} of $\fh$ is a Lie subalgebra $\fg$ of $\fh$ such that the complexification of $\fg$ is isomorphic to $\fh$, i.e. $\fg_\Cset \cong \fh$, so that every $X \in \fh$ can be uniquely written as $X= X_1 + i X_2$ for some $X_1, X_2 \in \fg$.  It is always possible to realize $\fg$ as the fixed point set of a conjugate-linear involution $\tau_*$ preserving the bracket on $\fh$:
$$
\tau_*: \fh \rightarrow \fh,  \qquad \tau_*^2(X) = X, \qquad \tau_*(\alpha X) = \bar{\alpha} \tau_*(X), \qquad  X \in \fh, \alpha \in \Cset.
$$
Note that for an involution $\tau$ on $H$, the induced involution $\tau_*$ on $\fh$ denotes the differential at the identity $e \in H$, namely $\tau_* = d\tau_e$.  In general, $\fh$ may have several non-isomorphic real forms arising from different choices for $\tau$.  Then for $G$ denoting the fixed point set of $\tau$ in $H$, there is a corresponding Lie subalgebra of $\fh$, namely
$$
\fg = \{ X \in \fh \,|\, \tau_* X = X \}.
$$
If $\fg$ is semisimple, then it has a maximal compact subalgebra $\fk$.  This subalgebra may also be realized as the fixed point set of a (real linear) involutive automorphism on $\fg$.  Suppose $\sigma_*$ is such an involution so that
$$
\fk = \{ X \in \fg \,|\, \sigma_* X = X \}.
$$

Restricting our attention to a particular real form $\fg$ with maximal compact subalgebra $\fk$, one may complexify $\fk$ to $\fk_\Cset$ and extend $\sigma_*$ by complex scalars to $\sigma_\Cset$, to obtain the diagram in Figure \ref{realforms2}; lines between two sets indicate an involution defined on the larger set which fixes the smaller set.
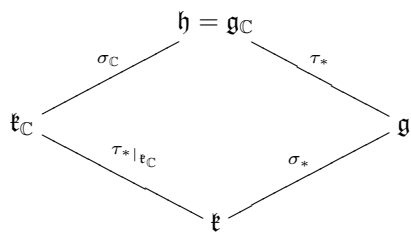
\begin{figure}[h]
\[
\xymatrixcolsep{4pc}
\xymatrix{
& \fh = \fg_\Cset  \ar@{-}[dl]^{}_{\sigma_\Cset}  \ar@{-}[rd]^{\tau_* }_{} & \\
\fk_\Cset   \ar@{-}[dr]^{{\tau_*}_{|_{\fk_\Cset}}}_{} && \fg  \ar@{-}[dl]^{}_{\sigma_*} \\
& \fk & 
}
\]
\caption{$\tau_*$ and $\sigma_\Cset$ are in bijective correspondence}\label{realforms2}
\end{figure}

\noindent  If $\fg$ is semisimple, it is known that there is a bijection between the conjugate-linear involutions $\tau_*$ and complex-linear involutions $\sigma_\Cset$.  Furthermore, the associated $\sigma \in \mbox{Aut } H$ commutes with $\tau$ (i.e. $\sigma\tau(h)=\tau\sigma(h)$ for all $h \in H$) \cite{Sahi}.  For a specific example, see Figure \ref{su22realforms}, Section \ref{supq}.

Next, define $\fp$ to be the $-1$ eigenspace of $\sigma_*$ in $\fg$, namely
$$ \fp = \{ X \in \fg \,|\, \sigma_* X = -X\}.$$
It is known (e.g. \cite{Barut}) that $G/K$ is a symmetric space, and that $\fg$ has (Cartan) decomposition $\fg = \fk + \fp$, with $[\fk,\fk] \subset \fk, [\fk,\fp] \subset \fp$, and $[\fp,\fp] \subset \fk$.  Since $G$ is semisimple, the Cartan-Killing form $$\eta(X,Y) = C^k_{il}C^l_{kj}X^i Y^j,\qquad X,Y \in \fg$$ is non-degenerate.  If $K$ is a maximal compact subgroup of $G$, then $\eta$ is negative definite on $\fk$ and positive definite on $\fp$.  Moreover,
 $\fk$ and $\fp$ are orthogonal subspaces with respect to $\eta$,
and $\fk$ is a maximal compact subalgebra of $\fg$.

\subsection{The Iwasawa Decomposition}\label{iwasdec}

Using the subspaces $\fk$ and $\fp$, we shall write down the Iwasawa Decomposition of $\mathfrak{g}$ and use it to establish a quadratic constraint on the symmetric space $G/K$.

Let $\fa$ denote a maximal subspace of $\fp$ that is an abelian subalgebra of $\fg$. The dimension of $\fa$ is called the {\em split rank} of $\fg$.    Let $\Delta_\fa$ be the root system of the pair $(\fg,\fa)$, i.e.
$$\Delta_\fa = \{ \lambda\in \fa^* \ |\ \la\ne 0,\fg_\fa^\la \ne 0\},\qquad \fg_\fa^\la = \{ X\in\fg\ |\ [X,Y] = \la(Y)X,\ \forall \,\,Y\in \fa\},$$
where $\fa^*$ is the dual vectorspace to $\fa$.  Then $\Delta_\fa$ is split by the Cartan involution, i.e. $\Delta_\fa = \Delta_\fa^- \cup \Delta_\fa^+$ and the involution maps one of these sets to the other one.  Introducing the nilpotent subalgebras $$\fn^\pm = \bigoplus_{\la\in \Delta_\fa^\pm} \fg_\fa^\la,$$
the Iwasawa decomposition of the $\fg$ is $\fg = \fn \oplus \fa \oplus \fk$, for $\fn = \fn^- \mbox{ or } \fn^+.$  It lifts  via the exponential map to a decomposition for the group $G=NAK$, where $K$ is the set of fixed points of $\sigma$ in $G$, and $A$, $N$ are the subgroups obtained by exponentiating the algebras  $\fa$ and $\fn^-$ (or $\fn^+$).

Let $\star$ denote the {\em twisted conjugation} induced by $\sigma$ on $G$, i.e. $g \star g' = g g' \sigma(g)^{-1}$ and let $S$ denote the orbit of the identity $e$ under $\star$, i.e. $S := \{ g \sigma(g)^{-1}\ | \ g \in G\}$.  Then $S$ is a {\em totally geodesic} submanifold of $G$ which is isomorphic to the symmetric space $G/K$ under the {\em isometric} embedding
\begin{eqnarray*}
\cC:G/K &\longrightarrow& G \\ 
gK &\longrightarrow& g\sigma(g)^{-1}.
\end{eqnarray*}
The mapping $\cC$ is known as the \emph{Cartan embedding} of the symmetric space into its Lie group \cite{Car27}.  Notice that if we view $\cC$ as a mapping from $G$ to $G$ (composing $g \mapsto gK$ with $\cC$) and suppose $q = \cC(g)$, then $q \sigma(q) = e$.  Thus the image of $G/K$ in $G$ under the embedding consists of elements satisfying a {\em quadratic constraint}: 
\beq\label{eq:star}
G/K = \{ q \in G\ |\  q\sigma(q) = e\}.
\eeq
Furthermore, by the Iwasawa Decomposition, for each $g \in G$, there exist unique $k \in K$, $a \in A$ and  $n \in N$ such that $g = nak$.  Thus, $\cC(g)$ can be expressed as
$$ \cC(g) = (nak)\sigma( (nak))^{-1} = nak \sigma( k^{-1} a^{-1} n^{-1})  = (na)k\sigma( k)^{-1} \sigma( na)^{-1}  = \cC(na),$$
so that the mapping $\cC$ is in fact well-defined and one-to-one on the solvable subgroup $S$ of $G$ consisting of elements $s\in G$ that can be written as $s=na$ for some $a\in A$ and $n\in N$; that $S$ is a subgroup follows from the fact that the commutator of $A$ and $N$ (i.e. elements of the form $nan^{-1}a^{-1}$) lies in $N$. 

Now since the two subspaces $\fk$ and $\fp$ of $\fg$ are orthogonal with respect to the Killing-Cartan form $\eta$, the restriction of $\eta$ to $\fp$ provides the symmetric space $G/K$ with a natural metric.  Since the Cartan embedding is totally geodesic, this metric agrees with the metric $\bh$ induced on $S=\C(G)$ as a submanifold of $G$ \cite{Hump72, Helgason-1, Kna86}.  Since the Cartan embedding kills the $K$-factor in the Iwasawa decomposition, it is possible to compute this metric for $S$ using only parameterizations of the subgroups $A$ and $N$ of $G$:  Given parameterizations $a(\bu)\in A$ and $n(\bv)\in N$, one computes first $s = na$ and $q = s \sigma(s^{-1})$, and then $w = q^{-1} dq$ (or $-dq q^{-1}$ as the case may be), from which $\bh$ can be computed.

\subsection{Harmonic maps into symmetric spaces}
Because of the Cartan embedding being totally geodesic, any harmonic map into $G/K$ is a harmonic map into $G$, and likewise any harmonic map into $G$ whose image is contained in the submanifold $G/K$ is a harmonic map into $G/K$.  Thus the task of constructing harmonic maps into a symmetric space, i.e. a solution of the nonlinear sigma-model, can be simplified by reducing it to finding a harmonic map into the corresponding Lie group, that is to say, finding a solution of the {\em principal chiral field model} (see Figure \ref{embeddings}).
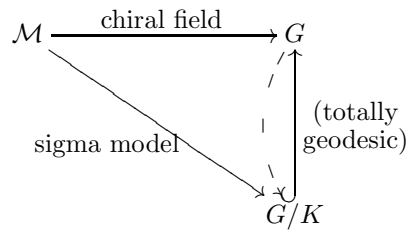
\begin{figure}[h]
\[
\xymatrixcolsep{3pc}
\xymatrix{
 \cM  \ar[rrdd]^{}_{\mbox{sigma model}} \ar[rr]^{\mbox{chiral field}}_{}  & & G \ar@{-->}@/_1pc/[dd]^{}_{}\\
& & \\
& & G/K \ar@{^{(}->}[uu]_{\stackrel{\mbox{(totally}}{\mbox{geodesic)}}} \\
}
\]
\caption{Harmonic maps as chiral field and sigma models}\label{embeddings}
\end{figure}

Furthermore, harmonic maps into symmetric spaces $G/K$ enjoy a large group of symmetries, since the full group $G$, which could be of considerably larger dimension than $G/K$, acts on it isometrically, thereby providing a large set of conserved currents for the harmonic map.  It has been suggested long ago \cite{EGK84, Xan84} that every field theory that can be formulated in terms of a harmonic mapping from an effectively two-dimensional domain manifold into a symmetric space, is completely integrable, and that the inverse-scattering technique can be utilized to generate new solutions from known ones.  We now establish this conjecture for axially symmetric harmonic maps.

\section{Integrability of axially symmetric harmonic maps}
Let $G \subset GL(n, \RR)$ be a semisimple Lie group, and let $K$ be a maximal compact subgroup of $G$.   Let $\cM = \RR^3$ have a Euclidean metric with line element $ds_\bg^2 = d\rho^2 + \rho^2 d\varphi^2 + dz^2$ given in cylindrical coordinates $(\rho,z,\varphi)\in \RR^{2}_+ \times \Sset^1$.  Suppose $f:\cM \to G/K$ is an axially symmetric harmonic map and let $q$ be a parametrization of the symmetric space via the Cartan embedding $\cC :G/K \to G$, so that $q = f(\rho,z)$.  By the discussion in Section \ref{liegroups}, the Maurer-Cartan form $w=-dq q^{-1} \in \bigwedge^1(G/K)$ has a corresponding pullback form $W = f^*w \in \bigwedge^1(\cM)$ which satisfies the Hodge system \refeq{hodgesys}.  The divergence of this axially symmetric 1-form $W$ in these coordinates is
$$\de W = - \frac{1}{\sqrt{|\det{\bg}|}}\p_\mu(\sqrt{|\det{\bg}|}\bg^{\mu\nu}W_\nu) = \frac{-1}{\rho}[\p_\rho(\rho W_\rho) + \p_z(\rho  W_z)],$$
so that the equation $\de W = 0$ is equivalent to $d(\rho *_2 W) = 0$ where $*_2$ is the Hodge star on $\RR^2_+$, i.e. $(*_2 W)_\rho = -W_z$ and $(*_2W)_z = W_\rho$.   Since $W_\phi = 0$ and the other components of $W$ are independent of $\phi$, we can also replace  $d$, the exterior derivative  on $\RR^3$ in these equations, by $d_2$ the exterior derivative on $\RR^2_+$.  Dropping the ``2" subscripts altogether, we rewrite the Hodge system \refeq{hodgesys} as
\beq\label{eq:W} dW + W\wedge W = 0,\qquad d(\rho *W) = 0,\eeq
where the domain is now understood to be the right-half plane in $\RR^2$ with coordinates $\bx =(\rho,z)$.

\subsection{Lax system on a Riemann surface bundle}
The goal is to describe a Lax pair, that is to say, an over-determined linear system of equations, for which the integrability condition is the $W$-system \eqref{eq:W} above.  We begin with the equation $W = -dq\  q^{-1}$ and rewrite it as 
\beq\label{dq}
dq = -Wq,
\eeq
noting that we now view $q$ as $q(\bx)$.  Following  \cite{ZB-I, Ale80, EGK84}, we generalize \refeq{dq} to the following linear system
\beq\label{Psisys}
\left\{\begin{array}{lll} D\Psi & =&  - \Omega \Psi\\
\left.\Psi\right|_{\la = 0} & = & q\end{array} \right.
\eeq
for the unknown $\Psi: \Cset\times \RR^2_+ \to \Cset^{n\times n}$.  Here, $D$ and $\Omega$ are generalizations of $d$ and $W$, respectively:
\beq\label{Psisyscond}
D_\mu := \p_\mu - \om_\mu\frac{\p}{\p \la},\quad \om_\mu := \frac{1}{\p\varpi/\p\la }\p_\mu\varpi \qquad \Omega_\mu := a W_\mu + b \rho (*W)_\mu,\qquad \mu = 1,2.
\eeq
The parameter $\la$ appearing in \eqref{Psisys} is in the Riemann Sphere $\bar{\Cset}$, and $\varpi(\la,\bx)$, $a(\la,\bx)$, $b(\la,\bx)$ are three $\bar{\Cset}$-valued functions on $\bar{\Cset} \times \RR^2_+$.  This particular form of the Lax system using  three functions  was first considered in \cite{EGK84}.  The functions $\varpi$, $a$ and $b$, are assumed to be rational  in $\la$, with coefficients that are smooth in $\bx$, and are subject to further restrictions.  Although it is possible to proceed at the level of generality appearing in \refeq{Psisyscond} for quite a while, in the interest of clarity we restrict our attention to specific choices for these functions.  In particular, $\varpi, a$ and $b$ shall be chosen in such a way that $D$ can be viewed as a covariant derivative (on an appropriate space) and that $D$ agrees with $d$ on $\la = 0$.

We turn our attention first to the function $\varpi$.  For $\bx= (\rho,z)\in\RR^2_+$ let $\cR_\bx$ denote the Riemann surface which is the zero-set of the quadratic polynomial $$F_\bx(\la,\varpi) := \la^2 -2\la (z-\varpi) - \rho^2,$$
 i.e., $\cR_\bx = \{(\la,\varpi)\in \Cset^2\ |\ F_\bx(\la,\varpi) = 0\}$.   For $\rho \ne 0$, this is a non-singular Riemann surface, with branch points at $\varpi = z \pm i\rho$, away from which $\varpi$  is  a two-to-one function of $\lambda$ given by
 \beq\label{varpi}
 \varpi(\la,\bx) = \frac{\rho^2}{2\la} + z - \frac{\la}{2}.
 \eeq
Equivalently, for a fixed $\varpi$, there are two charts on $\cR_\bx$, corresponding to the two roots $\la_\bx(\varpi), \la'_\bx(\varpi)$ of the quadratic in $\la$, each defined on a slit plane $\Cset \setminus \{z+it\ |\ -\rho\leq t\leq \rho\}$, $\la_\bx$ taking its values inside the disk $|\la|<\rho$ and $\la'_\bx$ outside of it, since $\la_\bx \la'_\bx = - \rho^2$.  There are two points over $\varpi=\infty$, and thus the two-point compactification $\overline{\cR}_\bx$ of $\cR_\bx$ is the Riemann sphere $\bar{\Cset}$.  The mapping $T_\bx:\bar{\Cset} \to \bar{\Cset}$, $T_\bx(\la) = -\rho^2/\la$ is a deck transformation on the universal cover of $\cR_\bx$ and $T^2=id$; since $\cR_\bx$ is parabolic, it is known that the universal cover is the complex plane $\Cset$. 
  Note  that by construction, $\left.\om\right|_{\la=0} = 0$.

With this notation in place, we make some remarks concerning the various bundles which are involved in the study of $D \Psi = -\Om \Psi$  in \refeq{Psisys}.  

First note that the domain of the solution $\Psi$ is precisely the Riemann surface bundle $\displaystyle{ \cB=\!\!\bigcup_{\bx \in\RR^2_+}\!\!\cR_\bx}$. Since the target of $\Psi$ is a Lie group $H$, the tangent bundle of the target has a canonical trivialization, which one associates with the Maurer-Cartan form taking values in $\fh$.  Using the pull-back under $\Psi$ of this form to $\cB$, the operator $D$ can also be viewed as a connection on the {\em pull-back bundle} $\Psi^{-1}TH$, whose fibers are isomorphic to the Lie algebra $\fh$, and the one-form $\Om$ can be viewed as a section  of this pullback bundle.  The Maurer-Cartan equations in the target give rise to zero-curvature equations for $\Omega$ in the domain.  We would like to realize $\Om$ as the (pulled-back) Maurer-Cartan form of some group element.  From this perspective, we have a zero curvature result analogous to Theorem 2.1 of~\cite{Uhl89}:
\begin{thm}\label{thm2.1} Let $H\subseteq GL(n, \Cset)$ be a complex Lie group with Lie algebra $\fh$.  Let $\cU$ be a simply connected domain in $\Cset \times\Rset^2_+$, let $\Om= A(\la,\bx)d\rho +B(\la,\bx)dz$ be a smooth $\fh$-valued 1-form defined on $\cU$, and suppose $D$ is defined as in \refeq{Psisyscond}.  Then the equation
$$D\Psi = -\Om \Psi$$ for $\Psi : \cU \to H$ has a solution iff the curvature of the connection $\Om$ vanishes, namely
\beq\label{zerocurvature}
D\Om + \Om \wedge \Om =0.
\eeq
\end{thm}
\begin{proof}   Assume there exists $\Psi$ such that $D_j\Psi = -\Om_j \Psi$, $j=1,2$.  Using the definitions of $D$ and $\om$ in  \eqref{Psisyscond}, it is easy to directly verify that $[D_1,D_2]=0$.  As a result, 
$$0=D_1D_2\Psi - D_2D_1\Psi = \left( D_2 \Om_1 -D_1\Om_2+ \Om_2\Om_1 - \Om_1 \Om_2 \right)\Psi .$$
Thus, the curvature of the connection vanishes:  $D\Om + \Om \wedge \Om=0$.  

Conversely, let  $\Om(\la, \bx) = A d\rho + B dz$ be a connection 1-form, with  $A,B:\cU \rightarrow \fh$, such that \refeq{zerocurvature} holds.  We first observe that $\Psi$ is to be constructed as a mapping from the Riemann surface bundle $\cB$ into the group $H$.  Denote by $\tilde{\Om}$ the pull-back of the Maurer-Cartan form on $\fh$ to the bundle $\cB$, noting that it satisfies the equation $d\tilde{\Om} + \tilde{\Om} \wedge \tilde{\Om} =0$.  We claim that under the appropriate choice of coordinates on each fibre $\cR_\bx$ of $\cB$, this equation is, in fact, the zero curvature equation \refeq{zerocurvature}.  In particular, we see that the operator $D$ can be expressed in terms of $d$, when $\varpi$ is chosen as the coordinate on $\cR_\bx$: Given $\Psi(\la, \bx)$, define $\tilde{\Psi}(\varpi, \bx)=\Psi(\la(\varpi,\bx),\bx)$.  Then
$$
\partial_\mu \tilde{\Psi}(\varpi, \bx)=\frac{\partial}{\partial \la}\Psi(\la(\varpi,\bx),\bx) \partial_\mu \la + \partial_\mu \Psi = -\frac{\partial_\mu \varpi}{\partial_\la \varpi} \frac{\partial}{\partial \la}\Psi(\la(\varpi,\bx),\bx) + \partial_\mu \Psi = D_\mu\Psi(\la, \bx).
$$
Consequently, the zero curvature condition \refeq{zerocurvature} can be rewritten for $\tilde{\Om}(\varpi, \bx):=\Om(\la, \bx)$ as
$$
d\tilde{\Om} + \tilde{\Om} \wedge \tilde{\Om} = 0.
$$
At this point, we may appeal to the standard zero-curvature theorem for $\tilde{\Om}$ (see, for instance, Cor. 4.24, p.~81 in \cite{Guest08}), to conclude there exists a mapping $F: \cU \rightarrow H$ such that $\tilde{\Om}=-dF F^{-1}$ (and consequently $\Om = -DF F^{-1}$) locally.  Since $\cU$ is simply connected, this statement holds globally.  Calling the global mapping $\Psi$, the converse direction is proved.
\end{proof}
We note that the mapping $F$ above will not be unique, since one may equally well consider its right-translation by a constant group element $h \in H$:  $(D(Fh))(Fh)^{-1} = (DF) hh^{-1} F^{-1}  = (DF) F^{-1}$.  

Finally, we would like to recover $dq=-Wq$ in \refeq{dq} from \refeq{Psisys} by further restricting the functions $a$ and $b$ such that  $\left.\Om\right|_{\la= 0} = W$.  In conjunction with the zero curvature condition, the system of equations for $a,b$ in \eqref{Psisyscond} are thus found to be~\cite{EGK84}:
\beq\label{abconds} 
\left.a\right|_{\la = 0} = 1,\quad \left.b\right|_{\la = 0} = 0,\quad a^2 + \rho^2 b^2 - a = 0, \qquad Da = \rho *Db.
\eeq
For subsequent sections, it will be useful to fix $a$ and $b$.  We do so now, in the form of a Lemma.
\begin{lem}\label{lem:cute}
Let $\Omega_\mu := a W_\mu + b \rho (*W)_\mu$,  $\mu = 1,2$ and define $a, b$ to be
$$a(\la,\bx) = \frac{\rho^2}{\la^2 + \rho^2},\qquad b(\la,\bx) = \frac{\la}{\la^2+\rho^2}.$$
Suppose $W$ a 1-form on $\RR^2_+$, then the zero-curvature condition \refeq{zerocurvature} is satisfied if and only if the Hodge system \refeq{eq:W} is satisfied.
\end{lem}
\begin{proof}
One may check directly that $a, b$ defined in the Lemma satisfy the necessary requirements stated in \refeq{abconds}.  Expanding \eqref{zerocurvature} using the definition of $\Om$, we obtain
$$a(dW + W\wedge W) + b \,d(\rho *W) = 0,$$
since $W$ is independent of $\la$.  For fixed $\bx$, this equation must hold for all $\la$.  Since $a$ and $b$ are linearly independent functions of $\la$, both equations in \refeq{eq:W} must hold simultaneously.
\end{proof}

\subsection{Gauge freedom}
By the general definitions of $\om, \varpi$ in \refeq{Psisyscond}, it is easy to see that the covariant derivative $D$ has the property that it vanishes on $\varpi$, and consequently on any sufficiently smooth function of $\varpi$:
\begin{prop}\label{prop:easy}
The kernel of the operator $D$ consists of arbitrary matrix-valued $C^1$ functions of $\varpi(\la,\bx) = \frac{\rho^2}{2\la} +z - \frac{\la}{2}$.
\end{prop}
\begin{proof}
This follows from Implicit Function Theorem.  Observe that  $F(\rho,z,\la,\varpi) :=\rho^2 + 2\la (z-\varpi) - \la^2 = 0$ and thus away from the branch points $\varpi = z\pm i\rho$, where $\p_\la F = 0$, we have that $\la = \la(\rho,z,\varpi)$ is $C^1$ and $\p_\rho \la = -\frac{\p_\rho F}{\p_\la F} = - \frac{\p_\rho \varpi}{\p_\la \varpi}$, and similarly  $\p_z \la = -\frac{\p_z \varpi}{\p_\la \varpi}$.  Let $f: \Cset\times \RR^2_+ \to \Cset$ be a $C^1$ function with $Df = 0$.  Then $g(\rho,z,\varpi) := f(\rho,z,\la(\rho,z,\varpi))$ is also $C^1$, by chain rule $\p_\rho g = \p_\rho f + \p_\la f \p_\rho \la = 0$ and similarly $\p_z g = 0$.  Thus $g$ is only a function of $\varpi$.  Applying this to each entry of a matrix $C \in \ker D$ establishes the result.
\end{proof}
 Thus, if $\Psi$ solves \refeq{Psisys}, then so does 
\beq\label{gauge}
\Psi'(\la,\bx) = \Psi(\la,\bx) C(\varpi(\la,\bx)),
\eeq
where $C:\overline{\Cset}\to G$ is {\em any} matrix-valued complex curve such that $\displaystyle{\lim_{\la \rightarrow 0}C(\varpi(\la)) = C(\infty)=I}$.  Visibly, the converse is also true.
\begin{cor}\label{cor:gauge}
Suppose $\Psi,\Psi':\Cset\times \RR^2_+ \to \Cset^{n\times n}$ satisfy the same linear equation
$$ D\Psi = -\Om \Psi,\qquad D\Psi' = -\Omega \Psi'.$$
Then there exists a matrix-valued complex curve $C:\Cset \to \Cset^{n\times n}$ such that \eqref{gauge} holds.
\end{cor}
\begin{proof}
Let $V = \Psi^{-1} \Psi'$. Then $DV = -\Psi^{-1} D \Psi \Psi^{-1} \Psi' + \Psi^{-1}D \Psi' = \Psi^{-1}\Om \Psi' -\Psi^{-1}\Om \Psi'=0$. By Proposition \ref{prop:easy}, there exists a $C^1$ matrix function $C(\varpi)$ such that $\Psi'=\Psi C$.
\end{proof}
This is the so-called {\em gauge freedom} in the initial value problem \refeq{Psisys}.   We denote by $[\Psi]$ the equivalence class of $\Psi$ under gauge transformations, so that $\Psi' \in [\Psi]$ iff there exists a map $C:\overline{\Cset}\to G$ with $C(\infty) = I$ such that \eqref{gauge} holds.  Note that the results of Proposition \ref{prop:easy} and Corollary \ref{cor:gauge} are not exclusive to our (fixed) choice of $\varpi$.

\subsection{Reality conditions}
As seen in the above, the Lax system \refeq{Psisys} has a plenitude of solutions, not all of which may be of interest to us or indeed useful for the purpose of  ISM.   One may ask for example if there are solutions $\Psi(\la)$  that remain in the real group $G$ for $\la \ne 0$.  To address this question, it is necessary to define (following \cite{TerUhl04, Terng10}) the concept of {\em $G$-reality}.

\begin{defn}\label{domain}  Let $\cD$ be a domain in $\Cset$ containing the origin that is invariant under complex conjugation, i.e. $\overline{\cD} = \cD$.  Let $G$ be a real form of the complex Lie group $H$, consisting of elements in $H$ that are fixed by the involutive automorphism $\tau:H \to H$. A mapping $g:\cD \to H$ is said to satisfy the $G$-{\em reality condition} if $$ \tau(g(\overline{\la})) = g(\la)\quad\mbox{ for all } \la \in \cD.$$
Similarly, a mapping $\xi :\cD \to \fh$ into the Lie algebra of $H$ is said to satisfy the $G$-reality condition if $\tau_*(\xi(\overline{\la})) = \xi(\la)$.
\end{defn}
With this definition we now have:
\begin{prop}\label{prop:psireal}
 There exists a domain $\cD$ as in Definition \ref{domain} such that for every solution $\Psi$ of \refeq{Psisys} there is a complex curve $C:\Cset \to \Cset^{n\times n}$ for which
\beq\label{realityPsi}
\tau(\Psi(\overline{\la}))=\Psi(\la)C(\varpi), \qquad \mbox{for all } \la \in \cD.
\eeq
Furthermore, $C$ satisfies the quadratic reality constraint 
\beq\label{eq:star3}
C(\overline{\varpi}) \tau(C(\varpi))=I.
\eeq
\end{prop}
\begin{proof}
One notes that for any fixed $\bx$, $\Omega = a W + \rho b *W$ is a Lie-algebra-valued mapping $\la \mapsto \Omega(\la) \in \fh$ (recall $W \in \fg$).  The functions $\varpi(\la)$, $a(\la)$ and $b(\la)$ are equivariant under conjugation, namely $\varpi(\overline{\la}) = \overline{\varpi(\la)}$ and similarly for $a$, $b$.  Since $\tau_*$ is conjugate-linear, this implies that $\Omega$ satisfies the $G$-reality condition
$$\tau_*(\Omega(\overline{\la})) = \Omega(\la). $$
From here it is easy to see that $D( \tau(\Psi(\overline{\la}))\Psi^{-1}(\la) ) = 0$, noting that the conjugate inside the expression is evaluated after differentiation.  By Proposition \ref{prop:easy}, we conclude that there exists $C(\varpi)$ such that  \refeq{realityPsi} holds.  Taking limits as $\la \to 0$, we observe that
$$
\lim_{\la \to 0}\tau(\Psi(\overline{\la}))=\tau(q)=q=\lim_{\la \to 0}\Psi(\la)C(\varpi)=q\,C(\infty),
$$
and therefore $C(\infty)=I.$  This means that there exists a neighborhood of $\la = 0$ for which $C$ is invertible.  On this neighborhood, apply the involution $\tau$ to the equation and substitute the original expression for $\tau(\Psi)$ to obtain
$$
\Psi(\overline{\la}) = \tau(\Psi(\la))\tau(C(\varpi))=\Psi(\overline{\la})C(\overline{\varpi})\tau(C(\varpi)),
$$
from which we can immediately deduce the quadratic reality constraints $C(\overline{\varpi})  \tau(C(\varpi))=I$.
\end{proof}

\subsection{Involutive symmetry}
We now move on to the consequences of the initial data $q$ of \refeq{Psisys} belonging to $G/K \hookrightarrow G$.  Since elements of the symmetric space satisfy the quadratic constraint \refeq{eq:star}, we have a corresponding symmetry for $W= -dq q^{-1}$ and hence also for $\Om$:
$$Wq + q \sigma_*(W) = 0, \qquad \Om(\la) q + q \sigma_*(\Om(\la))=0.$$
The above now imply a certain involutive symmetry for  $\Psi$  under the inversion $T_\bx$.
\begin{prop} \label{prop:psisym}
 There exists a domain $\cD$ as in Definition \ref{domain} such that for every solution $\Psi$ of \refeq{Psisys} there is a complex curve $J:\Cset \to \Cset^{n\times n}$ for which
\beq\label{symPsi}
\Psi(T_\bx(\la),\bx) =  q(\bx) \sigma( \Psi(\la,\bx) )J(\varpi(\la,\bx)), \qquad  \mbox{for all } \la \in \cD.
\eeq
Furthermore, $J$ satisfies the quadratic symmetry constraint 
\beq\label{eq:star2}
\sigma(J(\varpi)) J(\varpi)=I.
\eeq
\end{prop}
\begin{proof}
First, observe that  the functions $a$ and $b$ chosen in Lemma \ref{lem:cute} are {\em equivariant} under $T_\bx$, i.e.
   $$a(T_\bx(\la),\bx) = 1- a(\la,\bx),\qquad b(T_\bx(\la),\bx) = -b(\la,\bx),$$
 as a result of which, $\Om$ is equivariant under $T_\bx$ as well
 $$\Om(T_\bx(\la,\bx),\bx) = W - \Om(\la,\bx).$$
 
Let us then define $\tilde{\Psi}(\la,\bx) := \Psi(T_\bx(\la),\bx)$.  It is easy to check that $D$ commutes with $T_\bx$, i.e. $D\tilde{\Psi}(\la,\bx) = (D\Psi)(T_\bx(\la),\bx)$, since  $\varpi$ is invariant under  $T_\bx$.
It then follows that $\tilde{\Psi}$ satisfies the equation $D\tilde{\Psi} = (-W+\Om)\tilde{\Psi}$, and using the above mentioned symmetries of $\Omega$ and $W$, we observe that the function $\Psi' := q\sigma(\tilde{\Psi})$
 satisfies the same equation as $\Psi$, namely $D\Psi' = - \Om \Psi'$ (although not necessarily with the same initial value as $\Psi$).  Thus by Corollary~\ref{cor:gauge} we must have a complex curve $J$ for which $\Psi = \Psi' J(\varpi)$.  In other words there exists a mapping $J = J(\varpi)$ such that (suppressing dependence on $\bx$)
\beq\label{invsymPsi} \Psi(\la) = q \sigma (\Psi(T_\bx(\la)) ) J(\varpi(\la)),\eeq
for all $\la \in \Cset$.  Replacing $\la$ with $T_\bx(\la)$ in the above yields (recall that $\varpi$ is invariant under the inversion $T_\bx$)
\beq\label{syminvPsi} \Psi(T_\bx(\la)) = q \sigma( \Psi(\la) ) J(\varpi(\la)).\eeq
Substituting \refeq{syminvPsi} into \refeq{invsymPsi}, we conclude that $J(\varpi)$ must satisfy the quadratic constraint, and have proved the proposition.
\end{proof}
Note that $J(\varpi)$ encodes the asymptotic behavior of $\Psi$ as $\la \to \infty$: Since $T_\bx(\la)=-\rho^2/\la$, we may take the limit $\la \to 0$ in \eqref{syminvPsi} to conclude
\beq\label{psiasymp}
\lim_{\la \to \infty} \Psi(\la,\bx) = q \si(q) J(\infty) = J(\infty),
\eeq
which also shows that $\Psi$ at $\la = \infty$ is a constant map.

\subsection{Inverse scattering and vesture for harmonic maps}
Suppose that $q_0:\cM \to G/K$ is a given axially symmetric harmonic map into the symmetric space $G/K$, $q_0 = q_0(\rho,z)$.  Assume $\Psi_0$ be any solution to \eqref{Psisys} with initial data $q_0$, let $W_0 = -dq_0 \ q_0^{-1}$ and $\Om_0 = a W_0+b\rho *W_0$.  {\em Vesture} (dressing) refers to the possibility that other solutions $q = f(\rho,z)$ to the harmonic map may be generated by finding a matrix $\chi = \chi(\la,\bx)$ with the property that $\Psi$, defined by
\beq\label{ves}
\Psi = \chi \Psi_0,
\eeq
solves the PDE in \refeq{Psisys}.  If in addition we require $\Psi$ and $\Psi_0$ to have the same asymptotic behavior as $\la \to \infty$, it would follow that  
\beq\label{chiatinf}
\chi(\infty,\bx) = I,\qquad \forall \,\, \bx \in \RR^2_+.
\eeq
If such a matrix $\chi$ can be found, then setting $\la =0$, one obtains 
\beq\label{dressing}
q(\bx) = \chi(0,\bx) q_0(\bx),
\eeq
which gives us a new solution $q(\bx)$ to the harmonic map problem, obtained by ``dressing" the seed solution $q_0(\bx)$.  The process of arriving at $q$ from $q_0$ via $\chi$ is thus formally similar to the {\em inverse scattering} discussed in the Introduction.  Compare Figures \ref{classical}, \ref{vacuum}, and \ref{ismq}.
\begin{figure}[h]
\[
\xymatrixcolsep{5pc}
\xymatrix{
q_0(\bx)	\ar@{-->}[d] \ar[r]^-{D\Psi=-\Omega \Psi}_-{\Psi |_{\lambda=0} \,= q_0}	&\Psi_0(q_0, \lambda) \ar[d]^{\Psi = \chi \Psi_0} \\
q(\bx)									& \Psi(q, \lambda) \ar[l]^-{\lambda = 0}_{}
}
\]
\caption{Vesture for Harmonic Maps}\label{ismq}
\end{figure}

We shall now outline the procedure for finding the dressing matrix $\chi$ and generating new harmonic maps $q$ via $\Psi$.  Choose an ``initial seed"  map $q_0$, an axially symmetric harmonic map of $\RR^3$ into $G/K$, and solve the linear system \refeq{Psisys} to obtain an initial matrix solution of the Lax system, $\Psi_0$.  Let $$\Om_0 := a W_0 + b \rho *W_0$$ where $W_0 := - dq_0 q_0^{-1}$.  Now given {\em any} invertible matrix-valued map $\chi = \chi(\la,\bx)$ which is rational in $\la$ and smooth in $\bx$, we may {\em define} the one-form $\Om$ by
\beq\label{eq:Dchi}
\Om := (D\chi -\chi \Om_0)\chi^{-1}.
\eeq
By construction, then, 
\beq\label{eq:chi}
D\chi  = \chi \Om_0 - \Om \chi.
\eeq
We are of course only interested in those $\chi$ for which $q(\bx)=\Psi(0, \bx)$, the result of the dressing \refeq{dressing}, still belongs to the symmetric space $G/K$.  In particular, $\chi$ needs to satisfy symmetries corresponding to those of  $\Psi$, e.g. $\det \Psi(\la)=1$ for all $\la \in \cD$.  
 On the other hand, using the definition of $\Om$,
$$D(\det \chi) = \det \chi \mbox{ tr}( \chi^{-1} D\chi ) =  -\det \chi \mbox{ tr}( \chi^{-1} \left[ \chi\Om_0-\Om \chi \right] ) =  - \det \chi \mbox{ tr}\Omega.
$$
\emph{A priori},  tr $\Om$ does not vanish, so we cannot conclude that $\det \chi = 1$.  To address this, we make the following modifications (this is analogous to the modifications in \cite{ZB-I, ZB-II}.)  Set
\beq\label{primes}
\Om' = \Om - \frac{1}{n} \mbox{tr }\Om \,I_{n\times n}, \qquad \chi' = (\det \chi)^{-1/n} \chi,
\eeq
so that $\tr \Om' = 0$ and $\det \chi' = 1$.  Observe that the following still holds:
\beq
D\chi' = \chi'\Om_0-\Om' \chi'.
\eeq
This means that $\Psi'=\chi' \Psi_0$ solves \refeq{Psisys} with $\Om'$ in place of $\Om$. Since  tr $\Om'=0$ implies that $\Om' \in \fh$, we now have that $\Psi' \in H$ and thus $q(\bx)=\Psi'(0,\bx) \in H$.  

In order to further ensure that $q \in G/K \hookrightarrow G$ and therefore it is indeed a new map into the original symmetric space, further restrictions need to be imposed on $\chi$, namely the $G$-reality and involutive symmetry conditions corresponding to \refeq{realityPsi} and \refeq{symPsi}.  In the first case,
$\tau(\chi(\overline{\la})) = \chi (\la)$.
For the second condition, we simply observe that if $\chi(\la)$ is a solution to \refeq{chiatinf}--\refeq{eq:chi}, so will be $\chi'(\la) = q \sigma( \chi(T_\bx(\la))) \sigma(q_0 )$, where $q=\chi(0)q_0$.  We are thus going to require that $\chi' = \chi$ to ensure the resulting $\Psi$ possesses involutive symmetry as well. 

The above restrictions will guarantee the map $q$ is in the symmetric space $G/K$.  In order for $q$ to be a harmonic map, however, still further restrictions are needed, this time on the pole-structure of $\chi$ as a rational function of $\lambda$.   To summarize,
\begin{defn}\label{chidef} A mapping $\chi:\overline{\Cset}\times \RR^2_+ \to GL(n,\Cset)$ is a {\em dressing matrix} for $q_0$, an axially symmetric harmonic map of $\RR^3$ into $G/K$, if it satisfies all of the following:
\begin{eqnarray}
 \chi(\infty) &=& I \label{asympchi}\\ 
  \tau(\chi(\overline{\la})) &=& \chi(\la) \label{realitychi}\\
   \chi(\la) &=& q \sigma (\chi(T_\bx(\la)))\sigma( q_0), \qquad \mbox{for } q:=\chi(0)q_0. \label{invsymchi} 
\end{eqnarray}
In addition, the poles of $\chi$ are to be restricted in such a way that  the poles of $\Om$ defined by \refeq{eq:Dchi} are precisely those of $\Om_0$, and with the same residue at each pole.
\end{defn}
We now have the following theorem.
\begin{thm}\label{thm2.2}
If $q_0$ is an axially symmetric harmonic map of $\RR^3$ into $G/K$ and $\chi$ is a dressing matrix for $q_0$, then $q=\chi(0)q_0$ is also an axially symmetric  harmonic map of $\RR^3$ into $G/K$.
\end{thm}
\begin{proof}
Given $q_0, \chi$, construct $\Om$ as in \refeq{eq:Dchi}.  Modify $\Om$ to $\Om'$ as above, noting that this modification cannot introduce any new poles.  Given $a, b$ defined in Lemma \ref{lem:cute}, set
$$W:=\Om' - \frac{b}{a}\rho (*\Om') .$$
With this definition of $W$, one has $$\Om' = a W + b\rho *W.$$
We now claim that, for fixed $\bx\in \RR^2_+$, $W(\cdot,\bx)$ is holomorphic on the Riemann sphere, and is therefore constant in $\la$, by Liouville's Theorem.
  Recall that $\Om_0$ is holomorphic on the Riemann sphere except for simple poles at $\la = \pm i \rho$. From the definition of $W$ it is clear that the only possible poles for $W$ can be at $\la = \pm i\rho$.  Recall also that the residue of the meromorphic one-form $\alpha$ on the Riemann surface $R$ at a point $p$ is defined as
  $$ \Res_{p} \alpha = \Res_{z=p} f = \frac{1}{2\pi i}\int_{\gamma} \alpha,$$
  where $\alpha = f dz$ in local coordinates about $p$, and $\gamma$ is a small circle around $p$.
Calculating the residue of $W$ at its only two poles, $\la = \pm i\rho$, we obtain   
$$ \Res_{\la = i\rho} W = \Res_{\la = i\rho} \Om' - i (\Res_{\la = i\rho} *\Om') =  \Res_{\la = i\rho} \Om - i (\Res_{\la = i\rho} *\Om) ,$$
since $\frac{b}{a}\rho$ is holomorphic at $\la=\pm i \rho$.  Next, from the definition of $\chi$ we know that the poles of $\Om_0$ and $\Om$ agree (with the same residues), thus
$$\Res_{\la = i\rho} \Om = \Res_{\la = i\rho} \Om_0 = W_0 \Res_{\la = i\rho} a + \rho *W_0  \Res_{\la = i\rho} b = \frac{\rho}{2i} W_0 + \frac{\rho}{2} *W_0.$$
The final two equalities follow from the fact that $W_0$ is independent of $\la$ as well as using the definitions of $a$ and $b$, respectively.  Combining this result with the definition of $\Om'$, we have
$$\Res_{\la = i\rho} *\Om = \frac{\rho}{2i} *W_0 - \frac{\rho}{2} W_0 = -i \Res_{\la = i\rho} \Om,$$
i.e. the residue of $W$ at $\la = i\rho$ vanishes. Repeating the calculation now for $\la =-i\rho$ we conclude that $W$ has no poles anywhere on the Riemann sphere, which establishes the claim.

We thus have $\Om'=aW+ b \rho  (*W)$, where $W$ is now independent of $\la$.  Since $\Psi=\chi' \Psi_0$ is a solution to the Lax system \refeq{Psisys} by construction with $\Om'$ in place of $\Om$, we may apply Theorem \ref{thm2.1} to conclude that  $\Om' $ has zero curvature.  By Lemma~\ref{lem:cute}, this implies $W$ satisfies the Hodge system \refeq{eq:W}.  Thus $q$ is indeed a harmonic map. The symmetries of $\chi$ on the other hand ensure that $q \in G/K$, and we are done.
\end{proof}

\subsection{Algebraic solution of the vesture problem}
In general, the task of finding a dressing matrix $\chi$ for a given $q_0$ is seen to be equivalent to a Riemann-Hilbert problem in the complex plane \cite{ZB-I}, the complete solution of which requires solving a certain system of integral equations coupled to algebraic equations.  In some cases, however, referred to in the literature as ``the solitonic case", a special Ansatz for $\chi$ can be employed which guarantees that the Riemann-Hilbert problem has a trivial solution, and the aforementioned system then reduces to a purely algebraic system of equations.  Such an Ansatz was considered  in \cite{ZB-I,ZB-II,EGK84}, etc., postulating $\chi$ to be a rational function of $\la$ with a number of prescribed simple poles:
\begin{equation}\label{eq:Chi-Ansatz}
\chi(\la) = I + \sum_{k=1}^{2N} \frac{R_k}{\lambda - \la_k}.
\end{equation}
We say an $n\times n$ matrix $A$ has a simple pole at $\la = \la _0$ if $(\la-\la_0)A$ is non-zero and analytic at $\la = \la_0$.  The poles $\la_k=\la_k(\bx)$ are prescribed according to a certain recipe given below, and $R_k = R_k(\bx)$ are matrices to be found subsequently.  This Ansatz reduces the linear system \refeq{eq:chi} to a set of algebraic equations for a collection of singular matrices $M_k$, to be defined shortly.  We are assuming that a solution $\Psi_0(\la)$ of the Lax system \refeq{Psisys} corresponding to a harmonic map $q_0$ is given, and that the poles $\la_k$ being prescribed are {\em not} already poles of $\Psi_0$.  Therefore the new solution $\Psi$ constructed by way of the dressing matrix $\chi$ {\em will} have poles at $\la_k$.  The example we will consider later shows that this procedure will lead to an axisymmetric harmonic map with {\em ring singularities}.

We also observe that if $\chi$ has a pole at $\lambda = \la_k$, then by \refeq{invsymchi}, $\chi$ must also have a pole at $\la = T_\bx(\la_k)$, while by  \refeq{realitychi}, $\chi^{-1}$ must have a pole at $\lambda = \overline{\la}_k$. 
The above mentioned symmetries also suggest that poles $\la_k(\bx)$ are in fact being prescribed {\em on the Riemann surface} $\cR_\bx$, in the following sense:  Let $\{\varpi_k\}_{k=1}^N$ be $N$ distinct non-real complex numbers\footnote{Our setup does not allow for these poles to be real.  We should mention however, that there is a closely related setup which does allow the poles to be real, generalizing the approach in \cite{ZB-I}.  This is in particular important in the context of the Einstein Equations, if one insists on obtaining solutions which have black holes, as opposed to naked singularities. We will pursue this line of inquiry in a future paper.}; without loss of generality, we may choose these values to be in the upper half plane.   Define the poles
\beq\label{lanumbers}
\la_k = \la_\bx(\varpi_k),\qquad \la_{N+k} = \la'_\bx(\varpi_k),\qquad k = 1,\dots,N , 
\eeq
where $\la_\bx,\la'_\bx$ are the two charts for $\cR_\bx$. In particular, the poles come in pairs related by the deck transformation on $\cR_\bx$: $$\la_{N+k} = T_\bx(\la_k).$$   Moreover, the definition of $\la_k$ in terms of the two charts $\la_\bx$ and $\la_\bx'$ implies that $\varpi(\la_k(\bx),\bx) = \varpi_k$, which upon differentiation with respect to $\bx$ implies $\frac{\p\varpi}{\p \la} d_\bx \la_k + d_\bx \varpi = 0$, in other words,
\beq \label{lakom}
d \la_k + \om(\la_k(\bx),\bx) = 0.
\eeq
Now, since $\chi\chi^{-1} = I$ holds for all $\la$, it also holds at $\la = \la_k$, which is a pole of $\chi$.  It follows that $\chi^{-1}(\la)$ must be holomorphic at $\la=\la_k$ and non-zero, and further, that it is singular since its determinant has to vanish as $\la \to \la_k$.  Thus,
\beq\label{Rkissing}
R_k \chi^{-1}(\la_k) = 0,
\eeq
and in particular, the matrices $R_k$ must also be singular ($\det R_k = 0$).  We may use the symmetry properties of $\chi$ to find what they imply for the $R_k$.  For the sake of brevity, let us denote
$$\Psi^{-1}_k(\bx) := \lim_{\la \to \la_k(\bx)}\Psi^{-1}(\la,\bx),$$
and similarly for $\Psi_{0k}$  and $\chi^{-1}_k$.

First we observe that the following symmetry of $\chi$ is implied by \refeq{realitychi} and \refeq{invsymchi} (again, suppressing the $\bx$ dependence):
\beq\label{eq:chisym1}
\chi^{-1}(\la) q = q_0 \si (\tau [  \chi^{-1}(T_\bx(\overline{\la})) ]).
\eeq
Multiplying on the left by $R_k$ and taking the limit as $\la \to \la_k$, we  
use the numbering convention in \refeq{lanumbers} to obtain
\beq\label{Req1}
0=R_k q_0 \si\tau[ \chi^{-1}(T_\bx(\overline{\la}))]_{\la= \la_k} = R_k q_0 \si\tau[ \chi^{-1}(\overline{\la})]_{\la=\la_{N+k}},
\eeq
where the addition in $N+k$ is mod $2N$, i.e. for $k>N$, $N+k = k-N$.  Another symmetry of $\chi$, from \refeq{realitychi}, is
\beq\label{eq:chisym2}
\chi^{-1}(\la) = \tau(\chi^{-1}(\overline{\la})),
\eeq
which upon left multiplication by $R_k$ and taking the limit $\la \to \la_k$  yields 
\beq \label{Req2}
0 = R_k\tau[\chi^{-1}(\overline{\la})]_{\la=\la_k}.
\eeq
Equations \eqref{Req1}-\eqref{Req2} appear to be an over-determined system, since there are two matrix equations for each $R_k$.  In order to be able to reduce the number of equations, we now make an assumption about the involutions $\tau, \si$, namely that they should be given by conjugation with respect to the same element.  
\begin{assump}\label{involutionassumption}
Assume that there exists an element $\Ga \in H$ with $\Ga^2 = I$, such that
\beq\label{eq:Involutionfixing}
\tau(g) = \Ga (g^{*})^{-1} \Ga, \qquad \si(g) = \Ga g \Ga.
\eeq
\end{assump}
We remark that although there are semisimple Lie groups possessing involutions which cannot be realized in this way, the assumption can also be rephrased in terms of \emph{inner equivalence classes} in order to identify those Lie groups which are allowable \cite{Adams, Sahi}.  In particular, any Lie algebra having a Dynkin diagram with no symmetries satisfies this restriction. If on the other hand, the Dynkin diagram possesses nontrivial symmetries, such is the case for example if $H= SL(n,\Cset)$, then there will be real forms of $H$ where the corresponding involution is not realizable as conjugation with respect to an element, e.g. $G=SL(n,\RR)$.  However, even in those cases, other real forms of $H$ may still satisfy Assumption 1.  As an example, we note that such a matrix $\Ga$ exists for the pseudo-unitary groups $SU(p,q)$, to be discussed in Section \ref{supq}.  Under this assumption, the overdetermined system  \refeq{Req1}-\refeq{Req2} can be re-expressed as
\begin{eqnarray}\label{eq:MatrixReq12}
R_k q_0  \left[  \chi(\overline{\la}_{N+k}) \right]^*  &=& 0 \\
R_k \Ga \left[ \chi( \overline{\la}_{k}) \right]^* \Ga &=& 0. 
\end{eqnarray}
Or, using the Ansatz in \refeq{eq:Chi-Ansatz},
\begin{eqnarray}
R_k q_0 + \sum_{j=1}^{2N} \frac{1}{{\la}_{N+k}-\overline{\la}_j} R_k q_0  R^*_j  &=& 0\label{R1}\\
R_k + \sum_{j=1}^{2N} \frac{1}{{\la}_k - \overline{\la}_j} R_k \Ga R^*_j \Ga  &=& 0,\label{R2}
\end{eqnarray}
where the addition in $N+k$ is mod $2N$, i.e. for $k>N$, $N+k = k-N$.

Now clearly if $R_{N+k} = H_k R_k q_0 \Ga$ for some invertible matrices $H_k$, then the two equations above would become equivalent.  We will use Proposition~\ref{prop:psisym} to show that it is always consistent to make such a choice.  For $k=1,\dots,2N$, define the matrix
\beq\label{def:Mk}
M_k := R_k(\bx) \Psi_{0k}(\bx),
\eeq
noting that by assumption, $\la_k$ is not already a pole of $\Psi_0$.   Recall that the matrices $R_k$'s are known to be singular, from which it follows that the $M_k$ are also singular matrices in $\Cset^{n\times n}$ for each $k$.  Indeed, by \eqref{Rkissing},
\beq\label{Mkissing}
M_k \Psi^{-1}_k = R_k \Psi_{0k} \Psi^{-1}_k = R_k \chi^{-1}_k = 0.
\eeq
The symmetries of $\Psi$, \refeq{symPsi} in particular, together with (\ref{eq:star}) and (\ref{eq:star2}) imply 
$$\Psi^{-1}(\la,\bx) = \sigma[J(\varpi(\la))]\sigma[ \Psi^{-1}(T_\bx(\la),\bx)]\sigma(q).$$
Multiplying by $M_k$ and taking the limit $\la \to \la_k$ we obtain 
$$
M_k \sigma (J(\varpi_k))\sigma( \Psi_{N+k}^{-1} ) \sigma(q) = 0 \quad \Rightarrow \quad M_k \Ga J(\varpi_k) \Psi_{N+k}^{-1} \Ga =0.
$$
On the other hand, by \refeq{Mkissing} we also know that 
$$M_{N+k} \Psi_{N+k}^{-1} = 0.$$
Comparing the two equations above, matrices $M_{N+k}$ and $M_k \Ga J(\varpi_k)$ have the same null space, and thus it is consistent to assume that
\beq\label{eq:assumpM}
M_{N+k} = H_k M_k \Ga J(\varpi_k)=H_kM_k \sigma(J(\varpi_k))\Ga
\eeq
 for some invertible matrices $H_k$.  Note that as $\la \rightarrow \la_k$, the complex curve $J(\varpi)$ appearing in Proposition \ref{prop:psisym} also satisfies $J(\varpi_{N+k})J(\varpi_k)=I_{n\times n}$.  Consequently, this yields the relation between $R_{N+k}$ and $R_k$, making the two equations in (\ref{eq:MatrixReq12}) coincide, since the symmetry of $\Psi_0$ in \eqref{symPsi} implies
$$
R_{N+k} = M_{N+k}\Psi_{0,N+k}^{-1}  = H_k M_k \Ga J(\varpi_k) \Psi_{0,N+k}^{-1}  = H_k M_k \Psi_{0,k}^{-1}q_0 \Ga = H_k R_k q_0 \Ga,
$$
as claimed.  It is thus enough to keep only the equation \refeq{R2}, and rewriting it in terms of $M_k$ we arrive at the following (nonlinear) system
\beq\label{eq:theMksys1}
M_k \Psi_{0k}^{-1} \Ga + \sum_{j=1}^{2N} \frac{1}{{\la}_k - \overline{\la}_j} M_k S_{kj} M_j^* = 0,\qquad k=1,\dots,2N,
\eeq
where we have set 
\beq\label{def:Skj}
S_{kj} := \Psi_{0k}^{-1}\Ga (\Psi_{0j}^*)^{-1}=\Ga\Psi_{0k}^*\Ga \Psi_{0j} \Ga.
\eeq
It remains to ensure that the matrix $\chi$ thus defined has the correct pole structure as in Definition~\ref{chidef}.
In order to do so, we simplify \refeq{eq:theMksys1} further using an additional reduction:
We ssume that the $R_k$ (and hence the $M_k$) are {\em rank-one} matrices, i.e. there exists non-zero vector functions $\bu_k,\bv_k$ for $k=1,\dots,2N$ such that
$$ M_k = \bu_k \bv_k^*.$$
This rank-one assumption is consistent with $\det \chi$ having only {\em simple} poles at the $\la_k$, and is equivalent to assuming that the matrices $\chi_k^{-1}$ have rank $n-1$.

From (\ref{eq:assumpM}) we observe that it would then be consistent to take
$$ \bv_k = \Ga \bv_{N+k},\qquad k = 1,\dots, N.$$
Thus a complete set of unknowns for the problem are the vectors $\{\bu_k\}_{k=1}^{2N}$ together with $\{\bv_k\}_{k=1}^N$, i.e. $3N$ vectors in $\Cset^n$,  instead of the $2N$ matrices $M_k$ in $\Cset^{n\times n}$. Together they satisfy
\beq\label{eq:uvsys}
\sum_{j=1}^{2N} \frac{1}{{\la}_k - \overline{\la}_j} \bv_k^*S_{kj} \bv_j \bu_j^* = - \bv_k^*\Psi_{0k}^{-1} \Ga
\eeq
 Moreover, as a consequence of Theorem \ref{thm2.2}, we have the following
\begin{thm}
The dressing matrix $\chi$ can be constructed from \eqref{eq:uvsys} in such a way that 
$$\Om := (D\chi - \chi \Om_0)\chi^{-1}$$
 has the exact same pole structure as $\Om_0$.  Furthermore, without loss of generality, vectors $\{\bv_k\}_{k=1}^N$ can be taken to be arbitrary constants, i.e. independent of $\bx$.
\end{thm}
\begin{proof}
Clearly, it is necessary for  $\Om$ to be holomorphic at $\la = \la_k$ for any $k$, and any $\bx \ne \bx_k$, where the points $\bx_k$ are where $\la_k = \pm i\rho$ (the poles of $a$ and $b$, and thus of $\Om_0$).  It is easy to see that for any choice of $\varpi_k = \al+i\beta$, setting $z=\al,\rho = |\beta|$ will yield such a point $\bx_k$.  Thus the analysis below can be carried out at all but finitely many points in $\RR^2_+$.

We  use the Ansatz \refeq{eq:Chi-Ansatz} to rewrite the right-hand side of \refeq{eq:Dchi} in terms of inverse powers of $\la-\la_k$ and perform a residue analysis.  It turns out, thanks to \refeq{lakom}, that the coefficient of $(\la-\la_k)^{-2}$ in fact vanishes at $\la = \la_k$, and we end up with
$$
\left.(D\chi - \chi\Om_0)\chi^{-1}\right|_{\la = \la_k} =\left[ -\Om_0 + \sum_{k=1}^{2N} \frac{1}{\la-\la_k} \left(dR_k  - R_k \Om_0\right)\right]_{\la = \la_k} \chi^{-1}(\la_k).
$$
Since $\Om_0$ is  holomorphic at $\la = \la_k$ (and away from $\bx=\bx_k$), we obtain that the coefficient of $(\la-\la_k)^{-1}$ also has to vanish at $\la = \la_k$. Thus
\beq\label{eq:dR}
\left[dR_k(\bx)  - R_k (\bx)\Om_0(\la_k(\bx),\bx)\right] \chi^{-1}(\la_k(\bx),\bx) = 0.
 \eeq
Suppressing the $\bx$-dependence for brevity, we have
$$ dM_k = dR_k \Psi_{0k} + R_k (D\Psi_0)(\la_k) = [dR_k -R_k \Om_0(\la_k)] \Psi_{0k}$$
and using the rank one assumption, \refeq{eq:dR} becomes
$$
0= dM_k \Psi_{0k}^{-1} \chi_k^{-1} = dM_k \Psi_k^{-1} = d\bu_k \bv_k^*\Psi_k^{-1} + \bu_k d\bv_k^*\Psi_k^{-1}.
$$
Now by \refeq{Mkissing},  $M_k \Psi_k^{-1} = 0$ implies that $\bv_k^*\Psi_k^{-1} = 0$, which in turn gives $d\bv_k^*\Psi_k^{-1} = 0$ from the above equation.  Since $\Psi_k^{-1} = \Psi_{0k}^{-1}\chi_k^{-1}$ is rank $n-1$, $d\bv_k$ has to be a multiple of $\bv_k$, i.e. one must have
\beq\label{eq:vk}d\bv_k = h \bv_k,\eeq
for some {\em scalar}-valued 1-form $h$.  Applying $d$ now,  we obtain 
$$ 0 = dh \bv_k^*+ h\wedge h \bv_k^* = dh \bv_k^*,$$
which, since $\bv_k$ are nonzero, implies $dh = 0$, and hence by Poincar\'e's lemma, $h = d\gamma$ for some function $\gamma$.  Thus the differential equation \refeq{eq:vk} can be solved to get $\bv_k^* = e^{\gamma}\bv_k^*(\bx_0)$, i.e., the vector functions $\bv_k$ are scalar function multiples of a fixed vector $\bv_{k0}$.  From \refeq{eq:uvsys} it follows that if $\bu_k$ are scaled by a factor of $e^{-\ga}$, then $M_k$ and the new solution $q$ constructed by this method will be independent of the function $\ga$,  and therefore without loss of generality one can assume that $\bv_k$'s are {\em constant} to begin with.
\end{proof}

The equations for $M_k$ \refeq{eq:theMksys1} now become a {\em linear} system for the unknown vector functions $\bu_k$ in terms of  the constant vectors $\bv_k$, using \refeq{def:Skj}:
$$\sum_j a_{kj} \bu^*_j = \bb_k^*,\qquad a_{kj} := \frac{1}{\la_k - \overline{\la}_j} \bv_k^*S_{kj}\bv_j,\qquad \bb_k^* := -\bv_k^* \Psi_{0k}^{-1} \Ga.$$
These equations can be written in matrix form as  
\beq \label{linalg} AU^* = B^*,\eeq
where $U=U(\bx)$ is the ($n\times 2N$) matrix whose columns are the vector functions $\bu_k$, $A = A(\bx) = (a_{kj})$ is a $2N\times 2N$ matrix function, and $B=B(\bx)$ is the $n\times 2N$ matrix whose columns are the vector functions $\bb_k$ defined above.

If the matrix $A$ can be shown to be invertible\footnote{As far as we know, this point has not been addressed in previous studies of the dressing technique.}, at least in a neighborhood in $\RR^2_+$, then the above system has a unique solution $U^* = A^{-1}B^*$ in that neighborhood.
From there, one can then calculate the matrices $R_k$ and the dressing matrix $\chi$, and setting $\la =0$, the new solution $q$ is found to be
\beq\label{eq:solq}
q(\bx) = q_0(\bx) - \sum_{k=1}^{2N} \frac{1}{\la_k(\bx)} \bu_k(\bx) \bv_k^* \Psi_{0k}^{-1}(\bx) q_0(\bx).
\eeq
We will see in the examples of the next section, that the matrix $A$ in general is {\em not} invertible everywhere in the domain $\RR^2_+$, and indeed the zero set of $\det A$ has a geometric significance for the dressed harmonic map $q(\bx)$.  However, we can prove that under appropriate conditions on the arbitrary vectors $\bv_k$, the matrix $A(\bx)$ is in general invertible {\em for large }$|\bx|$, i.e. in a neighborhood of infinity.  
First we recall the following:
\begin{defn} 
A matrix $A \in \Cset^{n\times n}$ is called {\em strictly diagonally dominant} if $$\sum_{j\neq i} | a_{ij}| < |a_{ii}|,\qquad \mbox{ for all } i = 1,\dots,n.$$
\end{defn}
By the Levy-Desplanques Theorem (see e.g. \cite{Tau49}), the determinant of a strictly diagonally dominant matrix is non-zero.  We prove
\begin{thm}
Let $\bv_1,\dots \bv_{N} \in \Cset^n$ be $N$ arbitrary complex vectors satisfying the condition that, for all $k =1,\dots,N$,
\beq\label{cond:vk} \sum_{j\ne k} |\bv_k^*\Ga \bv_j |< \frac{1}{2} |\bv_k^*\Ga\bv_k|.\eeq
Set $\bv_{N+j} = \Ga \bv_j$ for $j=1,\dots,N$. 
Then, there exists $R>0$ such that for all $\bx \in \RR^2_+$ with $|\bx|>R$, the  matrix $A(\bx) \in \CC^{2N\times 2N}$ with elements $$ a_{ij} = \frac{1}{\la_i(\bx) - \overline{\la}_j(\bx)} \bv_i^*S_{ij}(\bx)\bv_j$$
where the $\la_i$'s are as in \refeq{lanumbers} and $S_{ij}$ as in \refeq{def:Skj},
is strictly diagonally dominant, and hence invertible.
\end{thm}
\begin{proof}
Let $\varpi_1,\dots,\varpi_n \in \Cset$ be $N$ distinct points in the upper-half complex plane, $\mbox{Im}(\varpi_j)>0$.  Set $$\varpi_j = z_j + i s_j.$$ For $j=1,\dots,N$ let  $(r_j,\theta_j)$ denote $N$ systems of elliptical coordinates, with real parameters $(z_j,s_j)$, on $\RR^2_+$. In terms of the Cartesian coordinates $(\rho,z)$ we have 
$$\rho =  \sqrt{r_j^2+s_j^2}\sin\theta_j,\qquad z = z_j + r_j\cos\theta_j.$$
Recall that $\la_j$ and $\la_{N+j}$ are the two roots of the quadratic polynomial $$\la^2 - 2(z-\varpi_j)\la -\rho^2 = 0.$$
In terms of the elliptical coordinates, we have
$$ \la_j = (r_j - i s_j) (\cos\theta_j +1),\qquad \la_{N+j} = (r_j +i s_j)(\cos\theta_j - 1).$$
Note that as $|\bx| \to \infty$, we have $r_j \to \infty$ for all $j$, and thus  $|\la_j| \to \infty$ as well.  

Fix $k \in \{1,\dots,N\}$.  Then
\beq\label{est:lak}|\la_k - \overline{\la}_k| = 2s_k(1+\cos\theta_k),\eeq
while, for $1\leq j\leq N$ and $j\neq k$,
\beq\label{est:laklaj} |\la_k - \overline{\la}_j| = \left| r_k(1+\cos\theta_k) - r_j(1+\cos\theta_j) - i\left( s_k(1+\cos\theta_k)+s_j(1+\cos\theta_j)\right)\right| \geq  s_k(1+\cos\theta_k).\eeq
For $N+1\leq j\leq 2N$ on the other hand,
\beq\label{est:laklajN}
|\la_k - \overline{\la}_j| =\left| r_k(1+\cos\theta_k) + r_j(1-\cos\theta_j) - i\left( s_k(1+\cos\theta_k)+s_j(1-\cos\theta_j)\right)\right| \geq r_k(1+\cos\theta_k).\eeq
It is clear that the same relations also hold for $k \in {N+1,\dots,2N}$, except that \refeq{est:laklaj} now holds for $N+1\leq j\leq 2N$ and \refeq{est:laklajN} holds for $1\leq j\leq N$.

Now recall that $\Psi_{0k}(\bx) = \Psi_0(\la_k(\bx),\bx)$. We have shown (Prop.~\ref{prop:psisym}) that, as $\la \to \infty$, the matrix $\Psi_0$ approaches a constant matrix $J$ that satisfies the quadratic constraint, see \eqref{psiasymp}. It thus follows from the definition of matrices $S_{ij}$ that $S_{ij} \to \Gamma $ as $|\bx| \to \infty$ for all $i,j=1,\dots,2N$, and hence
\beq\label{asymp:Skj} \bv_k^* S_{kj} \bv_j \to \bv_k^*\Ga\bv_j,\quad\mbox{ as }|\bx| \to \infty.\eeq
For vectors $\bv_j$ satisfying the condition \eqref{cond:vk}, and for $r_k$ large enough, we have
$$
\sum_{j=1,j\ne k}^N \frac{|\bv_k^*\Ga\bv_j|}{s_k} + \sum_{j=1}^N \frac{\bv_k^*\bv_j}{r_k} < \frac{|\bv_k^*\Ga\bv_k|}{2s_k}.
$$
Therefore, as $|\bx|\to \infty$, using \refeq{est:lak}-\refeq{est:laklajN},
$$\sum_{j\ne k} \frac{| \bv_k^*\Ga\bv_j|}{|\la_k - \overline{\la}_j|} < \frac{|\bv_k^*\Ga\bv_k|}{|\la_k - \overline{\la}_k|}.$$
Combining this with \eqref{asymp:Skj}, we arrive at the desired result, namely the diagonal dominance of $A$.
\end{proof}
We have thus shown
\begin{thm}\label{eq:mainthm}
Let $G$ be a real semisimple Lie group admitting a pair of commuting involutions which satisfy \eqref{eq:Involutionfixing} and suppose $K$ is a maximal compact subgroup of $G$.  Then the Hodge system \eqref{eq:W} for an axially symmetric harmonic map from $\mathbb{R}^3$ into the Riemannian symmetric space $G/K$ is integrable by way of ISM.  In particular, it is always possible to generate new harmonic maps from any given one using the dressing technique.
\end{thm}
 We note in passing that even though it appears that one needs to know the whole solution $\Psi_0(\la,\bx)$ to the Lax system \refeq{Psisys} with initial data $q_0$ in order to find the matrices $M_k$,  the only information one needs about $\Psi_0$ is its value {\em at the new poles} $\la = \la_k(\bx)$, i.e. the matrices $\Psi_{0k}(\bx)$.   These can be found by ``integration" (more precisely, exponentiation) from the seed solution $q_0(\bx)$, in the following manner:

It is straightforward to check that for any given  map $\phi(\la,\bx)$, the map $\psi(\bx) = \phi(\la_k(\bx),\bx)$ satisfies $(D\phi)(\la_k(\bx),\bx) = d\psi(\bx).$ In particular, the matrices $\Psi_{0k}$ satisfy the following equation
$$d\Psi_{0k} = - \Omega_{0k} \Psi_{0k} = - (a_k W_0 + \rho b_k *W_0) \Psi_{0k},$$
 where the $k$ index denotes evaluation at $\la = \la_k(\bx)$ as before, and $W_0$ we recall is the Maurer-Cartan form corresponding to the seed solution $q_0$.  Since $\Psi_0$ is an element of the group $G$, the Maurer-Cartan form corresponding to $\Psi_{0k}$,  $\mathcal{W}_k := -d\Psi_{0k} \Psi_{0k}^{-1}=a_k W_0 + \rho b_k *W_0$, is {\em known} once a seed solution and a set of poles $\varpi_k$ are specified.  Using the parametrization we have for the group, it is thus possible to recover $\Psi_{0k}$ from the Lie algebra element $\mathcal{W}_k$ by exponentiation.

\section{Applications: Generating Solutions}

\subsection{Agreement with previous results}
The techniques of this paper may be used to confirm calculations for stationary axisymmetric solutions to the Einstein Vacuum and Einstein-Maxwell equations.  We briefly indicate the approaches taken in these two particular cases of interest in gravitation before exhibiting a general approach which subsumes both.
The simple pole ansatz \eqref{eq:Chi-Ansatz} for $\chi$  appearing in this paper was introduced by \cite{ZB-I, ZB-II}, where stationary axisymmetry solutions to the Einstein vacuum equations were first studied.  The authors consider a two-dimensional representation of $SL(2,\mathbb{R})$, calculating $N$-soliton solutions explicitly.  Using the Minkowski metric  as an initial seed, the authors recover the Kerr-NUT metric in the 1-soliton case.  In contrast with the approach of this paper, the poles $\la_k$ appearing in $\chi$ are chosen in conjugate pairs and thus a {\em different} constant $\varpi_k$ is needed for each pole $\la_k$ .

Generalizing the approaches used in \cite{MaMi67, Misner} for the vacuum equations, the authors in \cite{EGK84} formulate the Einstein-Maxwell equations as a harmonic map into a symmetric space having isometry group $SU(2,1)$.  The choice of a 3-dimensional representation of this group realises the symmetric space as the complex hyperbolic plane $\mathbb{H}_\CC=SU(2,1)/S(U(2)\times U(1))$.  Although explicit examples of dressing do not appear in \cite{EGK84}, their discussion motivates our generalization to the complex Grassmann manifolds $SU(p,q)/S(U(p)\times U(q))$.

\subsection{New results: Noncompact Grassmann Manifolds}\label{supq}

We shall describe in detail a solution-generation method for the harmonic map equation, in the particular case where the domain is $\RR^3$ and the target is a {\em non-compact Grassmann manifold} $\cG_{p,q}$, a  $2pq$-dimensional Riemannian symmetric space realized as a quotient of the group $G=SU(p,q)$ by its maximal compact subgroup $K=S(U(p)\times U(q))$:
$$ \cG_{p,q} := SU(p,q)/S(U(p)\times U(q)).$$
  We restrict our attention further to axially symmetric solutions, the domain thus becoming effectively two-dimensional.  The results mentioned above on the Einstein vacuum and Einstein-Maxwell equations are thus about axisymmetric harmonic maps into $\cG_{1,1}$ and $\cG_{2,1}$, respectively.  As a further example, the study of the particular group $SU(2,2)$, and the associated Grassmann manifold $\cG_{2,2}$ is of physical interest, as this group is the universal (double) cover of $SO(4,2)$, the conformal group of  the Minkowski spacetime $\mathbb{R}^{3,1}$ \cite{Coq90, WehrBarut94}.  

We begin by recalling (e.g. \cite{Kna86}) that the real semi-simple Lie group $SU(p,q)$ can be identified with the subgroup of $SL(n,\Cset)$ consisting of matrices in $\Cset^{n\times n}$ preserving a pseudo-Hermitian quadratic form having $p$ plus signs and $q$ minus signs, where $p+q=n$:
$$G=SU(p,q)= \{ g \in SL(p+q,\Cset)  |  \Ga (g^*)^{-1} \Ga = g \}.$$ 
Here, $\Ga = \Ga_{p,q}$ is the $n \times n$ block-diagonal matrix $\diag(I_{p\times p}, -I_{q\times q})$.  Note that one may equally well choose any Hermitian matrix $\Ga$ conjugate to $\Ga_{p,q}$ to define the unitarily equivalent presentation of the group $SU(p,q)$, so for ease of notation, we suppress the subscripts $p,q$ on the matrix $\Ga$.  The corresponding Lie algebra to the group is
$$\mathfrak{su}(p,q) = \{ X \in \mathfrak{gl}(p+q,\Cset) \ |\ X^* \Gamma + \Gamma X = 0,\mbox{tr }X=0\}.$$

Define the following two involutions on $H=SL(p+q, \Cset)$
\beq\label{eq:tausigma}
\tau (g) = \Ga (g^*)^{-1} \Ga, \qquad \sigma (g) = \Ga g \Ga.  
\eeq
It is easy to see that $\tau$, $\sigma$ are involutive automorphisms, $SU(p,q) = \{ g \in SL(n, \Cset) \ | \ \tau(g) = g\}$, and $K = \{ g \in G\ | \ \sigma(g) = g\}$ is isomorphic to $S(U(p)\times U(q))$, which is a maximal compact subgroup of $G$, making $G/K$ a symmetric space.  The induced Lie algebra involutions are given by
\beq\label{eq:tausigmastar}
\tau_* X = -\Gamma X^* \Gamma, \qquad \sigma_* X = \Ga X \Ga, \qquad X \in \mathfrak{g}.
\eeq
The involution $\sigma_*$ complexifies to the complex-linear map $\sigma_\Cset(X)=\Ga X \Ga$ so that one has a diagram in Figure~\ref{su22realforms} analogous to that of Figure~\ref{realforms2}.
\begin{figure}[h]
\[
\xymatrixcolsep{3pc}
\xymatrix{
& \fg_\Cset = \mathfrak{sl}(n, \Cset)   \ar@{-}[dl]^{}_{\sigma_\Cset}  \ar@{-}[rd]^{\tau_* }_{} & \\
\fk_\Cset = \fs(\mathfrak{gl}(p)\times\mathfrak{gl}(q))   \ar@{-}[dr]^{\tau_{|_{\fk_\Cset}}}_{} && \fg=\fs\fu(p,q)  \ar@{-}[dl]^{}_{\sigma_*} \\
& \fk=\fs(\fu(p)\times \fu(q))& 
}
\]
\caption{$\tau_*(X)=-\Ga X^*\Ga$ ($\Cset$-linear ) and $\sigma_\Cset(X)=\Ga X \Ga$ ($\bar{\Cset}$-linear) are in bijective correspondence}\label{su22realforms}
\end{figure}

The induced involutions give rise to the Cartan Decomposition $\fg = \fk \oplus \fp$.  Here,
$\fp = \{ X \in \fg\ | \ \Ga X \Ga = -X\},$ and we observe that $\fk$, the Lie algebra of $K$, is  a maximal compact subalgebra of $\fg$, which is indeed isomorphic to $\fs(\fu(p) \times \fu(q))$.  The quadratic constraint on the image of $G/K$ under the Cartan embedding is
\beq
G/K  = \{ q \in G\ |\  q \Ga q \Ga = I \}.
\eeq
It is straightforward to confirm that the corresponding Lie algebra, $\mathfrak{su}(p,q) = \fg$ has dimension $(p+q)^2-1$, while $\dim\fk =p^2+q^2-1$, and thus $\dim \cG_{p,q}=\dim\fp=2pq$; further details in this case are carried out in \cite[p. 39]{Barut}. 

One also notes that since $G=NAK$ by the Iwasawa Decomposition, each $g\in G$ can be expressed in the form $g=nak$ for unique $k\in K$, $a \in A$, $n \in N$; consequently, for $G=SU(p,q)$ with the above defined involutions
$$g\sigma(g)^{-1}= na \sigma (a^{-1}) \sigma (n^{-1}) = na \Ga a^{-1} n^{-1} \Ga.$$
On the other hand, by the definition of $G$ and the fact that $\tau(g) = \Ga (g^*)^{-1} \Ga$, 
$$ g \sigma(g)^{-1} = g \Ga g^{-1} \Ga = gg^* = nakk^*a^*n^* = naa^*n^*,$$
since $k\in K$ is a unitary matrix.  Thus in particular the image of the symmetric space under the Cartan embedding consists of Hermitian matrices.  Moreover, having parametrizations for the subgroups $A$ and $N$ suffices for obtaining a parametrization of the symmetric space.  This will be carried out explicitly for $SU(2,1)$ in the next section.

A crucial fact about Grassmann manifolds is that the natural embedding $\cG_{p',q'} \hookrightarrow \cG_{p,q}$, for  $p'\leq p$, $q'\leq q$, is {\em totally geodesic}.  This is due to the following more general result (see \cite[Thm. IV.7.2]{Helgason-1}):
\begin{thm}\label{six}
Let $G$ be a semi-simple Lie group and  $K$ a compact subgroup of $G$.  Let $X=G/K$ and $X'=G'/K'$, where $G'$ is a subgroup of $G$ and $K' = G' \cap K$.   Suppose $\sigma$ is the involution on $G$ that fixes $K$ and $\sigma ' = \sigma |_{X'}$.  Then the natural embedding $X'\hookrightarrow X$ is totally geodesic.
\end{thm}
The proof consists of noting that the hypotheses imply Cartan decompositions $\fg = \fk + \fp$ and $\fg' = \fk'+\fp'$, and that $\fp'$ is thus a Lie triple system in $\fp$: $[\fp',[\fp',\fp']]\subset \fp$, which is equivalent to the embdding being totally geodesic.

Application of this theorem to the case at hand is immediate: Set $G=SU(p,q)$, $G'=SU(p',q')$, where $p'\leq p$ and $q'\leq q$.  Since any $g' \in G'$ satisfies ${g'}^* \Ga_{p',q'} g' = \Ga_{p',q'}$ it follows that the matrix representation of $g'$ has block-diagonal form, with a $p'\times p'$ block followed by a $q'\times q'$.  In order to obtain a corresponding element in $G$ it is thus enough to insert a copy of $I_{n-n',n-n'}$ between the two diagonal  blocks of 
$g'$.  In other words the embedding $\mathbf{i} : G' \to G$ is simply
$$
g' = \left(  \begin{array}{cc} g^+_{p'\times p'} & \\ & g^-_{q'\times q'} \end{array} \right) \mapsto g = \left(\begin{array}{ccc} g^+_{p'\times p'} & & \\ 
& I_{n-n',n-n'} \\ & & g^-_{q'\times q'} \end{array} \right).
$$
It is then clear that $g$ satisfies $g^*\Ga_{p,q}g = \Ga_{p,q}$ and is thus an element of $G$.   $K$ and $K'$, the maximal compact subgroups of $G$ and $G'$, are afforded by the fixed points of the Cartan involution given by $\sigma_{\Cset}(X)= \Ga X \Ga$, which agrees with $\sigma_*$ on $\mathfrak{su}(p,q)$.  That is to say, $K= S(U(p) \times U(q))$, $K' = S( U(p') \times U(q'))$.

With the above machinery in place, one is in a position to apply the results of Theorem \ref{eq:mainthm} directly, generating solutions for the harmonic map equation in this context.  We carry out the procedure explicitly in the subsequent sections for $N=1$ and a physically meaningful initial seed $q_0$.

\subsection{Einstein's vacuum equations: a 1-soliton calculation}
It is known \cite{MaMi67,Ern68a,Car73} that the Einstein vacuum equations $$\bR_{\mu\nu} = 0,$$ under the assumption of existence of two commuting non-null Killing fields $\bK,\tilde{\bK}$, reduce to the equations for an axially symmetric harmonic map $f :\cM^3 \to \cN^2$.  In the case one of the Killing fields is spacelike and the other timelike (outside a compact region in the spacetime,) the domain $\cM$ is the Euclidean space $\RR^3$.  When both Killing fields are spacelike the domain is the Minkowski space $\cM = \RR^{2,1}$, and in that case one usually calls $f$ a {\em wave map} instead of a harmonic map (since the equations are hyperbolic and in fact can be written as a semilinear system of wave equations). In either case the target  $\cN$ is the real hyperbolic plane $\Hset_\RR$, which is a symmetric space.  By the results of the main theorem, the system is therefore completely integrable, and new solutions can be found from old ones by the vesture method \cite{ZB-I}.   This fact can be established directly, by taking $G=SU(1,1)$, $K=U(1)$, and noticing $\Hset_\RR = G/K$.

Einstein's vacuum equations in the stationary axisymmetric case can be written as a system for $x$ and $y$, where for the spacetime metric $\bg$, $x:= \bg(\bK,\bK)$.  Thus $x>0$  since $\bK$ is assumed spacelike outside the axis of symmetry.  Let $\bb := i_\bK *d\bK$ be the \emph{twist form} for the Killing field $\bK$.  It is not hard to see that the Ricci-flatness of $\bg$ implies $d\bb=0$, and thus we may let $y$ denote a potential function for $\bb = dy$.  Let $(z,\rho,\theta)$ denote Euclidean cylindrical coordinates on $\RR^3$.  Then the equations satisfied by $x=x(z,\rho)$, $y=y(z,\rho)$ are exactly those for an axially symmetric harmonic map into the upper half plane model of $\Hset_\RR$.  Indeed, the image of the map under the Cartan embedding of $\Hset_\RR$ into $SL(2,\RR)$ is the following matrix
\beq \label{xyhyp} q(\rho,z) = \frac{1}{x} \left(\begin{array}{cc} x^2 +y^2 & y \\ y & 1\end{array}\right).\eeq

Now the Minkowski space $\RR^{3,1}$  is clearly a stationary, axisymmetric solution of the vacuum Einstein equations and as such, it is a natural choice for a seed solution in the ISM scheme.  It corresponds to the constant map $$x = 1,\qquad y = 0,$$ where we have let $\bK$ denote the timelike Killing field $\frac{\partial}{\partial t}$.  The corresponding element in the symmetric space $G/K$ is therefore the identity matrix, $$q_0 = q_0(\bx)=I_{2\times 2}:\RR^2_+ \longrightarrow  G/K=SU(1,1)/S(U(1)\times U(1)).$$ 
 Note that any other asymptotically flat solution of the Einstein vacuum equations must have the same behavior as the above map, in the limit $|\bx|\to \infty$.

Let $\{\varpi_k\}_{k=1}^N$ be a subset of $\Cset\setminus \RR$, and let $\la_k(\bx)$ be such that
$\varpi(\la_k(\bx),\bx)  = \varpi_k = \varpi(\la_{k+N}(\bx),\bx)$.  The set $\{\la_k(\bx)\}_{k=1}^{2N}$ is the set of new poles, for a solution $\Psi$ of the Lax system, which is to be found using the vesture method:  We have, trivially,  $\Psi_{0k}(\bx) = I_{2\times 2}$. 
   From here one finds the matrices $S_{kj}$ and the numbers $a_{kj}$.  One can then solve the linear system \refeq{eq:uvsys} to find the $\bu_k$'s in terms of the constants $\varpi_k$ and $\bv_k$, thus ending up with a $4N$-(real)parameter new solution $q(\bx)$ as in \refeq{eq:solq}.  We'll do this now for $N=1$:

Fix a complex number $\varpi = i s$, $s \in \Rset_+$ and let $\la_1, \la_2$ be the two roots of
$$p(\bx, \la):=\la^2 - 2(z-is)\la - \rho^2 = 0.$$
Visibly, $\Om_0= \mathbf{0}_{2\times 2}$ and by inspection, a simple corresponding solution $\Psi_0$ of the Lax system \refeq{Psisys} is $I_{2\times 2}$. 

 To solve \refeq{eq:theMksys1}, note that $S_{kj}=\Ga$ and make the choice $v_1=(\alpha,\delta)^t$, so that $v_2=\Ga v_1=(\alpha, -\delta)^t$, where $\alpha, \delta$ are arbitrary complex parameters.  It is possible to keep $v_1$ free, noting there will be relations between the parameters (e.g., symmetries, quadratic constraint).  For the purposes of this paper, the above assumption allows for simpler calculations.  The resulting linear system $AU^* = B^*$ is written as
\beq\label{linsys}
\left[ \begin{array}{cc} \frac{1}{\la_1 - \overline{\la_1}}(|\alpha|^2-|\delta|^2) &  \frac{1}{\la_1 - \overline{\la_2}}(|\alpha|^2+|\delta|^2) \\  \frac{1}{\la_2 - \overline{\la_1}}(|\alpha|^2+|\delta|^2) &  \frac{1}{\la_2 - \overline{\la_2}}(|\alpha|^2-|\delta|^2) \end{array} \right] \left[ \begin{array}{c} u_1^*  \\  u_2^*  \end{array} \right] = \left[ \begin{array}{cc} -\overline{\alpha} &  \overline{\delta} \\ -\overline{\alpha} &  -\overline{\delta} \end{array} \right].
\eeq
Upon solving for each of the vectors $u_1, u_2$, the harmonic map is given by
\beq \label{fnlsol} q(\bx) = I_{2\times 2} - \frac{1}{\la_1}u_1v_1^* - \frac{1}{\la_2}u_2v_2^*.\eeq
Anticipating the Kerr solution as the outcome, we change from Weyl coordinates $(\rho, z)$ to Boyer-Lindquist (oblate spheroidal) coordinates $(r, \theta)$ by setting
\beq\label{eq:oblate}
\rho = \sqrt{(r-m)^2+s^2} \sin \theta, \qquad z=(r-m)\cos \theta,
\eeq
for real parameters $m, s$.  Then the roots of $p(\bx, \la)$ are given by
$\la_{1,2}=[(r-m)\pm is][\cos \theta \mp1]$.  
Solving the linear system \eqref{linsys} we obtain
$$U = \frac{1}{2D} \left(\begin{array}{cc} 
\al ( \frac{iA}{s (\cos \theta +1)} - \frac{B}{r-m-is} ) & \al (  \frac{-iA}{s (\cos\theta -1)} + \frac{B}{r-m+is} ) \\
\de ( \frac{-iA}{s (\cos\theta +1)} - \frac{B}{r-m-is} ) & \de (  \frac{-iA}{s (\cos\theta -1)} - \frac{B}{r-m+is})
\end{array}\right)$$
where we have denoted 
$$ A := |\al|^2 - |\de|^2,\quad B :=  |\al|^2 + |\de|^2,\quad D:= \frac{B^2}{4((r-m)^2+s^2)} - \frac{ A^2}{ 4 s^2\sin^2\theta}.$$
Plugging this into \eqref{fnlsol} we arrive at
$$
q = \left(\begin{array}{cc}1 + \frac{8|\al|^2|\de|^2 s^2}{F} & \frac{1}{F}[-4s\al \bar{\de} (i A(r-m)-Bs\cos\theta)] \\
\frac{1}{F}[4s\bar{\al}\de (i A(r-m) + Bs\cos\theta)] & 1+\frac{8|\al|^2|\de|^2 s^2}{F} \end{array} \right)
$$
where $F:= A^2((r-m)^2+s^2)-B^2s^2\sin^2\theta$.  

The above (after adjusting its determinant to be one if necessary) is a harmonic map into the symmetric space $\cG_{1,1}$, for any value of the real parameters $m,s$ and complex parameters $\al,\de$.  Using the totally geodesic embedding of $\cG_{1,1}$ into $SU(1,1)$ we can view $q(\bx)$ as an element of $SU(1,1)$ and consequently of $SL(2, \Rset)$, after a Cayley transform:  Let $$Q := \frac{1}{\sqrt{2}} \left( \begin{array}{cc} 1 & i \\ i & 1\end{array}
\right).$$ Then $q' := QqQ^*\in SL(2,\RR)$ for $q\in SU(1,1)$, and thus, by \eqref{xyhyp}, the Ernst potential $\cE = x+ iy$ corresponding to a new stationary axisymmetric solution of vacuum equaions  can be obtained by setting 
$$x = \frac{1}{q'_{22}} = \frac{A^2((r-m)^2+s^2)-B^2 s^2 \sin^2\theta}{A^2((r-m)^2+s^2)-B^2 s^2 \sin^2\theta +2s^2(B^2-A^2) + 4sn_1A(r-m) -4s^2Bn_2\cos\theta}$$
and 
$$ y = \frac{q'_{12}}{q'_{22}} = \frac{4sAn_2(r-m) +4B n_1 s^2 \cos\theta}{A^2((r-m)^2+s^2)-B^2 s^2 \sin^2\theta +2s^2(B^2-A^2) + 4sn_1A(r-m) -4s^2Bn_2\cos\theta}
$$
where we have set $\al \bar{\delta} = n_1 + i n_2$.  It is now evident that if we make the following identifications:
$$ A = s,\qquad B = a := \sqrt{m^2+s^2},\qquad n_1 = \frac{m}{2}, \qquad n_2 = 0,$$
we would obtain
$$ x = \frac{r^2 - 2mr + a^2 \cos^2\theta}{r^2+ a^2 \cos^2\theta},\qquad y = \frac{2ma\cos\theta}{r^2+ a^2 \cos^2\theta}$$
which are precisely the expressions for the Kerr metric in Boyer-Lindquist coordinates.
We have thus shown that give any two positive values $m$ and $s$ it is possible to choose the parameters $\al$ and $\de$ in such a way as to recover the Kerr  metric with total mass $m$ and total angular momentum per unit mass $a  = \sqrt{m^2 +s^2}$, namely it is enough to set
$$\al = \sqrt{ \half (s + \sqrt{m^2+s^2})},\qquad \de = \sqrt{ \half (-s + \sqrt{m^2+s^2})}.$$
 Note that the Kerr solution thus obtained is necessarily naked, since $a>m$.  

It is likewise possible to obtain the {\em Kerr-Newman solution} to the Einstein-Maxwell equations as a 1-solitonic harmonic map into the complex hyperbolic plane $\Hset_\Cset$ by dressing the  trivial solution corresponding to the Minkowski metric, in the same manner as in the above, as will be demonstrated in the next section.

\subsection{The Einstein-Maxwell system}
It is quite remarkable that the Einstein-Maxwell system,
 $$
 \mathbf{R}_{\mu\nu} - \half \mathbf{g}_{\mu\nu}R = \kappa T_{\mu\nu};\qquad T_{\mu\nu} := F_\mu^\la *F_{\nu\la} - g_{\mu\nu} F_{\alpha\beta}F^{\alpha\beta};\qquad dF = 0;\qquad d*F = 0,
 $$
 which are the equations governing the interaction of the spacetime metric $\bg$ with an electromagnetic field $\bF$ permeating that spacetime, under the assumption of existence of two commuting Killing fields that also leave the field  invariant, still reduce to the equations for an axially symmetric harmonic map $f:\cM \to \cN'$, where $\cM$ is as in the above, and $\cN'$ is the {\em complex} hyperbolic plane $\Hset_\Cset$ \cite{Ern68b,Car73,Maz84}.  Moreover, this target is  a symmetric space, and indeed a non-compact Grassmann manifold $\cG_{2,1} \equiv \Hset_\Cset = G/K$ with $G=SU(2,1)$, and $K=S(U(2)\times U(1))$, so that once again, integrability is established and vesture method can be used to generate new solutions \cite{Ale81,EGK84}.

To follow in the steps of the last example, it is computationally advantageous  to make a change of basis and consider a unitarily equivalent representation of $SU(2,1)$.  In particular, we use
$$G = SU(2,1) = \{ g \in GL(3,\Cset) \ |\ g^*\tilde{\Ga} g = \tilde{\Ga}\},
\qquad\tilde{\Ga}:=\left[ \begin{array}{ccc} 0 & 0 & -i \\ 0 & 1 & 0 \\ i & 0 & 0 \end{array}\right]. $$ 
The involutions are correspondingly defined as in \eqref{eq:tausigma}, \eqref{eq:tausigmastar},
so that the Lie algebra of $G$ is given by  
$$\fg = \fs\fu(2,1)=\{ X \in \mathfrak{gl}(3,\Cset) \ |\ X^*\tilde{\Ga} + \tilde{\Ga} X = 0\}.$$
We find a basis for $\fg$ and confirm that it is eight-dimensional:  Indeed, $\fs\fu(2,1) = \mbox{span}\{X^1,\dots X^{8}\}$, where
$$
\begin{array}{llll}
 X^1 =  E_{13} & 
 X^2 = E_{31} &
 X^3 = E_{11}-E_{33} &
 X^4 = i (E_{11} -2E_{22}+E_{33}) \\
 X^5 = E_{12}+iE_{23}&
 X^6 =iE_{12}+E_{23}&
 X^7 = E_{21}-iE_{32} &
 X^8 = iE_{21}-E_{32}.
\end{array}
$$
Here $E_{ij}$ denote the members of the standard basis for $\Cset^{3\times 3}$ and $C^k_{ij}$ denote the structure constants of $\fs\fu(2,1)$, given in this basis by the commutation relations $[X^i,X^j] = C^k_{ij}X^k$.  The relations defining these constants are given in the table below.
\begin{table}[h]
$$
\begin{array}{c|ccccccc}
  & X^2 & X^3   & X^4   & X^5 & X^6 & X^7 & X^8 \\
\hline 
X^1 & X^3 & -2X^1  & 0 & 0 & 0 & -X^6 & X^5 \\
X^2 &   & 2X^2 & 0 & -X^8 & X^7 & 0 & 0 \\
X^3 &   &       &  0 & X^5 & X^6 & -X^7 & -X^8 \\
X^4 &   &       &       & 3X^6 & -3X^5 & -3X^8 & 3X^7 \\
X^5 &   &  (\mbox{anti-sym.})     &       &   & 2X^1 & X^3 & X^4 \\
X^6 &   &       &       &   &      & X^4 & -X^3 \\
X^7 &   &       &       &   &      &       & 2X^2 \\
\end{array}
$$
\caption{Commutation table for $\fs\fu(2,1)$ structure constants}\label{commtab}
\end{table}

We use Table \ref{commtab} to easily determine the $\pm 1$ eigenspaces of $\sigma$ in $\fg$ to be
\bea
\fk &=& \{ X \in \fg\ | \sigma_* X = X\} = \mbox{span } \{ X^1-X^2, X^4, X^5-X^7, X^6+X^8\}\\ 
\fp &=& \{ X \in \fg\ | \sigma_* X = -X\}= \mbox{span } \{ X^1+X^2, X^3, X^5+X^7, X^6-X^8\}.
\eea
It then follows (e.g. \cite[p. 39]{Barut}) that the Lie algebra $\fg$ has the Cartan decomposition $\fg = \fk \oplus \fp$, as described in section \ref{symmsp}.

To write down the Iwasawa Decomposition of $\mathfrak{g}$ in our basis, we choose
$$\fa = \mbox{span }\{X^3\},$$
which is clearly a maximal subspace of $\fp$ that is an abelian subalgebra of $\fg$.  For this $\fa$, we compute the root system to be as follows, where we have identified $\fa^*$ with $\Rset^1$, and $\langle\ \rangle$ denotes the linear span 
$$\begin{array}{cccc}
\fg_\fa^{-2} = \langle X^2\rangle,&
\fg_\fa^{-1} = \langle X^7, X^8\rangle,&
\fg_\fa^{1} = \langle X^5, X^6\rangle,&
\fg_\fa^{2} = \langle X^1\rangle.\\
\end{array}
$$
Accordingly, $\Delta_\fa^- = \{-2, -1, -1\}$ and $\Delta_\fa^+ = \{2, 1, 1\}$ (see section \ref{iwasdec}).  The Iwasawa decomposition of the Lie algebra is then chosen to be $\fg = \fn^+ \oplus \fa \oplus \fk$, which lifts  via the exponential map to a decomposition for the group $G=NAK$, where $K$ is the set of fixed points of $\sigma$ in $G$ and $A,N$ consisting of diagonal, unipotent elements, respectively, are the Lie groups obtained by exponentiating the algebras  $\fa, \fn^+$, respectively.

We derive parameterizations for the subgroups  $N$ and $A$ by exponentiating the corresponding Lie algebras.  For the particular 4-dimensional representation of $\fg$ that we have chosen, we have
$$\fa = \span\{ X^3\}, \qquad \fn^+ = \span\{ X^1, X^5 , X^6\}.$$
Therefore, an element $a\in A$ can be parameterized, using one real parameter $\mu$ as
$$a(\mu) = e^{\mu X^3} = \diag( e^\mu ,1 , e^{-\mu} ),$$
and similarly, $n\in N$ can be parameterized by one real parameter $\delta$ and one complex parameter $\eta + i \theta$:
$$n(\delta,\eta,\theta) = e^{\delta X^1 + \eta X^5 + \theta X^6 } = \mbox{exp} \left(\begin{array}{ccc} 0 & \eta + i \theta & \delta \\ 0 & 0 & i\eta + \theta \\ 0 & 0& 0 \end{array}\right)= \left(\begin{array}{ccc} 1 & \eta + i \theta & \delta + \frac{i}{2}(\eta^2 + \theta^2) \\ 0 & 1 & i \eta + \theta \\ 0 & 0 & 1  \end{array}
\right).$$
We can thus find an explicit parametrization for the image of $NA$ under the Cartan embedding, it will be of the form $q = q(\mu, \delta, \eta ,\theta)= na^2n^*$.  This image is isomorphic to the symmetric space $G/K$, as claimed.  For our purposes, however, it is useful to express this product in terms of the Ernst potentials
$$\sqrt{2} \Phi = \eta + i \theta \qquad \cE = e^{2\mu} + i \delta = x +|\Phi|^2 + iy. $$
In particular, an element of the symmetric space  is calculated to be
$$
\tilde{P}:=g \sigma (g)^{-1}=
\left(\begin{array}{ccc} 
x + 2 |\Phi|^2 + \frac{1}{x}(y^2 + |\Phi|^4) & \sqrt{2}\Phi(1-\frac{i}{2}y+\frac{1}{x}|\Phi|^2) & \frac{1}{x}(y+i |\Phi|^2) \\
\sqrt{2}\bar{\Phi}(1+ \frac{i}{2}y+\frac{1}{x}|\Phi|^2) & 1+\frac{2}{x}|\Phi|^2 & \frac{i\sqrt{2}}{x}\bar{\Phi} \\
\frac{1}{x}(y-i |\Phi|^2) & \frac{-i\sqrt{2}}{x}\Phi & \frac{1}{x}
\end{array}
\right).
$$
Notice that when $\Phi =0$, one recovers an embedding of $\cG_{1,1}$ in $\cG_{2,1}$, as expected.  We remark that the above representation of $\cG_{2,1}$ is unitarily equivalent to the one appearing in \cite{Chandrasekhar}, p. 571, under the identification of our variables $x, y, \Phi, \cE$ with $\Psi, \Phi, H, Z$, respectively.

We would like to return this matrix to the original representation of $SU(2,1)$ in terms of $\Ga$ (as opposed to $\tilde{\Ga}$), using conjugation by $Q^*$.  We thus obtain
\beq\label{eq:Pmatrix}
P:= Q \tilde{P} Q^* = \frac{1}{2|\Phi|^2- (\cE +\bar{\cE})}\left[ \begin{array}{ccc} 2|\Phi|^2 - |\cE|^2 -1 &2\Phi(1-\bar{\cE}) & -i(|\cE|^2-\cE+\bar{\cE}-1)\\  2\bar{\Phi}(1-\cE) & -(2|\Phi|^2 +\cE+\bar{\cE}) & -2i\bar{\Phi}(\cE+1)\\ i(|\cE|^2+\cE-\bar{\cE}-1) & 2i\Phi(\bar{\cE}+1) & -(2|\Phi|^2+|\cE|^2+1)  \end{array}\right].
\eeq
The role of $P$ is to suggest which constants may be chosen when applying the inverse-scattering mechanism to generate solutions.

We will now solve \refeq{eq:theMksys1} for $N=1$ in this case, choosing the Minkowski seed $q_0(\bx)=I_{3 \times 3}$, and noting that $S_{kj}=\Ga= \diag(1,1,-1)$.  Setting $v_1=(\alpha, \beta, \gamma)^t$, we have $v_2=\Ga v_1=(\alpha, \beta, -\gamma)^t$, where $\alpha, \beta, \gamma$ are arbitrary complex parameters.  The resulting linear system $AU^* = B^*$ becomes
\beq\label{linsys2}
\left[ \begin{array}{cc} \frac{1}{\la_1 - \overline{\la_1}}(|\alpha|^2+|\beta|^2-|\delta|^2) &  \frac{1}{\la_1 - \overline{\la_2}}(|\alpha|^2+|\beta|^2+|\delta|^2) \\  \frac{1}{\la_2 - \overline{\la_1}}(|\alpha|^2+|\beta|^2+|\delta|^2) &  \frac{1}{\la_2 - \overline{\la_2}}(|\alpha|^2+|\beta|^2-|\delta|^2) \end{array} \right] \left[ \begin{array}{c} u_1^*  \\  u_2^*  \end{array} \right] = \left[ \begin{array}{ccc} -\overline{\alpha} & -\overline{\beta} & \overline{\gamma} \\ -\overline{\alpha} & -\overline{\beta} & -\overline{\gamma} \end{array} \right].
\eeq
Our goal is to determine the appropriate choices for the complex constants using the matrix $P$ derived above.  We transform to Boyer-Lindquist coordinates to simplify the calculations, as before; note that the roots of $\la_{1,2}$ of $p(\bx, \la)$ are exactly as in the previous example.  It is easily verified that the linear system \eqref{linsys2} has solution.
$$U = \frac{1}{2D} \left(\begin{array}{cc} 
\alpha ( \frac{iA}{s (\cos \theta +1)} - \frac{B}{r-m-is} ) & \alpha (  \frac{-iA}{s (\cos\theta -1)} + \frac{B}{r-m+is} ) \\
\beta ( \frac{iA}{s (\cos \theta +1)} - \frac{B}{r-m-is} ) & \beta (  \frac{-iA}{s (\cos\theta -1)} + \frac{B}{r-m+is} ) \\
\gamma ( \frac{-iA}{s (\cos\theta +1)} - \frac{B}{r-m-is} ) & \gamma (  \frac{-iA}{s (\cos\theta -1)} - \frac{B}{r-m+is})
\end{array}\right)$$
where we have denoted 
$$ A := |\alpha|^2 +|\beta|^2 - |\gamma|^2,\quad B :=  |\alpha|^2 +|\beta|^2+ |\gamma|^2,\quad D:= \frac{B^2}{4((r-m)^2+s^2)} - \frac{ A^2}{ 4 s^2\sin^2\theta}.$$
Isolating each of the vectors $u_1, u_2$, the harmonic map is given by 
\begin{eqnarray*}\label{fnlsol21} 
q(\bx) &=& I_{3\times 3} - \frac{1}{\la_1}u_1v_1^* - \frac{1}{\la_2}u_2v_2^* \\
	  &=&  \left(\begin{array}{ccc}
1 + \frac{8|\alpha|^2|\gamma|^2 s^2}{F} & \frac{1}{F}[8\alpha\bar{\beta}|\gamma|^2s^2] & \frac{1}{F}[-4s\al \bar{\gamma} (i A(r-m)-Bs\cos\theta)] \\
q_{(2,1)} & 1 + \frac{8|\beta|^2|\gamma|^2 s^2}{F} & \frac{1}{F}[-4s\beta\bar{\gamma} (i A(r-m) - Bs\cos\theta)] \\
q_{(3,1)} & q_{(3,2)}  & 1+\frac{8(|\alpha|^2+|\beta|^2)|\gamma|^2 s^2}{F}
 \end{array} \right)
\end{eqnarray*}
Here, $F:= A^2((r-m)^2+s^2)-B^2s^2\sin^2\theta$.

We would like to show that the above six-parameter family of harmonic maps into $\cG_{2,1}$ contains as a special case, the three-parameter family of Kerr-Newman metrics in Boyer-Lindquist coordinates $(r,\theta)$.  To that end, we make the following identifications
$$ A= s \quad B = -a := \sqrt{m^2 + s^2 - e^2}.$$
This gives rise to a system of equations for the constants $\alpha, \beta$ and $\gamma$.  Setting $\alpha\bar{\gamma} = n_1 + i n_2$ and $\beta\bar{\gamma}= n_3+in_4$, one matches the 1-soliton solution generated from the Minkowski seed with the Kerr-Newman solution by choosing
$$ n_1 =\frac{m}{2} \quad n_2 =0 \quad n_3 = - \frac{e}{2} \quad n_4 =0.$$
One then obtains exactly the Ernst potentials for Kerr-Newman (see Eq. (21.26), p.~326 in \cite{Stephani}),
$$\Phi = \frac{e}{r-ia \cos \theta} \qquad \cE = 1- \frac{2m}{r-ia \cos \theta},$$
for real parameters $e, a, m$. 
 \subsection{Higher-dimensional vacuum gravity and beyond}
The $\cG_{p,q}$ nonlinear sigma model described in the previous section has found applications in other settings (e.g., \cite{NSanch82, Breitenlohner}), and in particular in the study of higher-dimensional gravity.  Moreover, explicating black-hole solutions in $d$-dimensional vacuum gravity for $d>4$ has been of recent interest both in (minimal) supergravity and in string theory \cite{EmparanReall, Virmani}.  In this context, stationary solutions possessing $d-3$ rotational Killing fields have been extensively studied; imposing an additional timelike Killing field results in effectively two-dimensional theories, to which integrability techniques apply.  

For instance, the authors in \cite{EmparanReall} consider the Einstein vacuum equations for $d=5$, having one timelike and two spacelike Killing fields, admitting a metric of the form
$$
ds^2 = g_{ab}(\rho, z) dx^a dx^b + e^{2 \nu(\rho, z)}(d\rho^2 + dz^2),
$$
where $\det g  = - \rho^2$.  The Einstein equations then reduce to
\bea
\partial_\rho U  + \partial_z V &=& 0 \\
U:= \left( \rho \partial_\rho g g^{-1} \right), &\quad& V:= \left( \rho \partial_z g g^{-1} \right)  \\
\partial_\rho \nu = -\frac{1}{2\rho} + \frac{1}{8\rho}\tr (U^2 - V^2), &\quad& \partial_z \nu = \frac{1}{4\rho}\tr (UV).
\eea
The first two equations of the system comprise a principal chiral field model into $GL(3, \RR)$, coupled to equations in $\nu(\rho,z)$ which may be solved by quadrature once $g$ is determined.  The Zaharov-Belinski technique is employed to produce Kerr, Myers-Perry and black ring solutions.  On the other hand, one could look for a sigma-model representation of the field equations (that is to say, a harmonic map into the symmetric space $SL(3, \RR)/SO(2,1)$) and employ the solution-generating methods of this paper to address the case of a non-diagonal initial seed.  This and other applications will be pursued in a forthcoming paper.

Further progress on this topic may also result from extending the integrability results of this paper to situations where the target is not a symmetric space but  a {\em homogeneous space}, in supergravity descriptions of M-theory, for example \cite{CERL05, CERL06}.  These ideas will be pursued elsewhere.

\subsection{Summary \& Outlook}
By establishing the integrability of harmonic maps from effectively 2-dimensional domains into Riemannian symmetric spaces $G/K$ for real semisimple Lie groups $G$, and showing how the dressing technique can be used to generate new solutions from known ones, our paper goes beyond the current literature in providing a general framework for the study of harmonic maps into \emph{noncompact} symmetric spaces commonly appearing in mathematical physics.  Examples we consider suggest that this approach may also make generating solutions for other effectively two-dimension geometric field theories more tractable. Future directions for research on this topic include the possibility of dressing with poles on the real line, which would give rise to a different class of singularities for the dressed solution (e.g. black holes vs. naked singularities), and extending the integrability results of this paper to more general targets, e.g. homogeneous spaces.

\section{Acknowledgements}
Both authors are indebted to Professor S. Sahi for illuminating discussions on Lie algebras and symmetric spaces, and to Professor M. Kiessling for reading a first draft of this paper and coming up with suggestions on how to improve the presentation.  SB gratefully acknowledges support from Professor M. Kiessling and the NSF through grant DMS-0807705, and thanks Dr. A. Virmani and the Max Planck Institute-Albert Einstein Institute for their reception during Summer 2012.  STZ thanks the Institute for Advanced Study for their hospitality and the stimulating environment provided during Spring 2011 while the authors were working on this project.

\bibliographystyle{plain}
\bibliography{SB_STZ_Aug28}

\def\cprime{$'$} \def\cprime{$'$}
  \def\polhk#1{\setbox0=\hbox{#1}{\ooalign{\hidewidth
  \lower1.5ex\hbox{`}\hidewidth\crcr\unhbox0}}}
\begin{thebibliography}{10}

\bibitem{AKNS-3}
M.~J. Ablowitz, D.~J. Kaup, A.~C. Newell, and H.~Segur.
\newblock Method for solving the sine-{G}ordon equation.
\newblock {\em Phys. Rev. Lett.}, 30:1262--1264, 1973.

\bibitem{AKNS-2}
Mark~J. Ablowitz, David~J. Kaup, Alan~C. Newell, and Harvey Segur.
\newblock Nonlinear-evolution equations of physical significance.
\newblock {\em Phys. Rev. Lett.}, 31:125--127, 1973.

\bibitem{AKNS-1}
Mark~J. Ablowitz, David~J. Kaup, Alan~C. Newell, and Harvey Segur.
\newblock The inverse scattering transform-{F}ourier analysis for nonlinear
  problems.
\newblock {\em Studies in Appl. Math.}, 53(4):249--315, 1974.

\bibitem{Adams}
Jeffrey Adams.
\newblock Strong real forms and the {K}ac classification.
\newblock Unpublished communication. http://www.liegroups.org/papers/, 1995.

\bibitem{Ale80}
G.~A. Alekseev.
\newblock N-soliton solutions of {E}instein-{M}axwell equations.
\newblock {\em Pis'ma Zh. Eksp. Teor. Fiz.}, 32(4):301--303, 1980.

\bibitem{Ale81}
G.~A. Alekseev.
\newblock On soliton solutions of {E}instein's equations in a vacuum.
\newblock {\em Dokl. Akad. Nauk SSSR}, 256(4):827--830, 1981.

\bibitem{Barut}
Asim~O. Barut and Ryszard R{\polhk{a}}czka.
\newblock {\em Theory of group representations and applications}.
\newblock World Scientific Publishing Co., Singapore, second edition, 1986.

\bibitem{ZB-I}
V.~A. Belinski{\u\i} and V.~E. Zakharov.
\newblock Integration of the {E}instein equations by means of the inverse
  scattering problem technique and construction of exact soliton solutions.
\newblock {\em Sov. Phys. JETP}, 48(6):985--994, 1979.

\bibitem{ZB-II}
V.~A. Belinski{\u\i} and V.~E. Zakharov.
\newblock Stationary gravitational solitons with axial symmetry.
\newblock {\em Sov. Phys. JETP}, 77(1):3--19, 1979.

\bibitem{Breitenlohner}
Peter Breitenlohner, Dieter Maison, and Gary Gibbons.
\newblock {$4$}-dimensional black holes from {K}aluza-{K}lein theories.
\newblock {\em Comm. Math. Phys.}, 120(2):295--333, 1988.

\bibitem{Car27}
Eli Cartan.
\newblock {La th\'eorie de groupes finis et continus et l'Analysis situs}.
\newblock {\em M\'em. Sci. Math. Fasc. XLII}, 1930.

\bibitem{Car73}
Brandon Carter.
\newblock Republication of: Black hole equilibrium states.
\newblock {\em General Relativity and Gravitation}, 41:2873--2938, 2009.
\newblock 10.1007/s10714-009-0888-5.

\bibitem{Chandrasekhar}
Subrahmanyan Chandrasekhar.
\newblock {\em The Mathematical Theory of Black Holes}.
\newblock Oxford University Press, New York, 1992.

\bibitem{Christodoulou}
Demetrios Christodoulou.
\newblock {\em The action principle and partial differential equations}, volume
  146 of {\em Annals of Mathematics Studies}.
\newblock Princeton University Press, Princeton, NJ, 2000.

\bibitem{CERL05}
Edmund~J. Copeland, James Ellison, Jonathan Roberts, and Andr{\'e} Lukas.
\newblock Isometries of low-energy heterotic {M} theory.
\newblock {\em Phys. Rev. D (3)}, 72(8):086008, 9, 2005.

\bibitem{CERL06}
Edmund~J. Copeland, James Ellison, Jonathan Roberts, and Andre Lukas.
\newblock Cosmological solutions of low-energy heterotic {M} theory.
\newblock {\em Phys. Rev. D (3)}, 73(8):086009, 19, 2006.

\bibitem{Coq90}
R.~Coquereaux.
\newblock Lie balls and relativistic quantum fields.
\newblock {\em Nuclear Phys. B Proc. Suppl.}, 18B:48--52 (1991), 1990.
\newblock Recent advances in field theory (Annecy-le-Vieux, 1990).

\bibitem{EicFor80}
H.~Eichenherr and M.~Forger.
\newblock More about nonlinear sigma models on symmetric spaces.
\newblock {\em Nuclear Phys. B}, 164(3):528--535, 1980.

\bibitem{EmparanReall}
Roberto Emparan and Harvey Reall.
\newblock Black holes in higher dimensions.
\newblock {\em Living Rev. Rel.}, 11(6):1--76, 2008.
\newblock arxiv.org/abs/0801.3471v2.

\bibitem{EGK84}
Ahmet Eri{\c{s}}, Metin G{\"u}rses, and Atalay Karasu.
\newblock Symmetric space property and an inverse scattering formulation of the
  {SAS} {E}instein-{M}axwell field equations.
\newblock {\em J. Math. Phys.}, 25(5):1489--1495, 1984.

\bibitem{Ern68a}
Frederick~J. Ernst.
\newblock New formulation of the axially symmetric gravitational field problem.
\newblock {\em Phys. Rev.}, 167:1175--1178, Mar 1968.

\bibitem{Ern68b}
Frederick~J. Ernst.
\newblock New formulation of the axially symmetric gravitational field problem.
  ii.
\newblock {\em Phys. Rev.}, 168:1415--1417, Apr 1968.

\bibitem{Virmani}
Pau Figueras, Ella Jamsin, Jorge~V. Rocha, and Amitabh Virmani.
\newblock Integrability of five-dimensional minimal supergravity and charged
  rotating black holes.
\newblock {\em Classical Quantum Gravity}, 27(13):135011, 37, 2010.

\bibitem{GGKM}
Clifford~S. Gardner, John~M. Greene, Martin~D. Kruskal, and Robert~M. Miura.
\newblock Korteweg-de{V}ries equation and generalization. {VI}. {M}ethods for
  exact solution.
\newblock {\em Comm. Pure Appl. Math.}, 27:97--133, 1974.

\bibitem{Guest08}
Martin~A. Guest.
\newblock {\em From quantum cohomology to integrable systems.}
\newblock Oxford: Oxford University Press, Oxford, UK, 2008.

\bibitem{GurXan82}
Metin G{\"u}rses and Basilis~C. Xanthopoulos.
\newblock Axially symmetric, static self-dual {${\rm SU}(3)$} gauge fields and
  stationary {E}instein-{M}axwell metrics.
\newblock {\em Phys. Rev. D (3)}, 26(8):1912--1915, 1982.

\bibitem{Helgason-1}
Sigurdur Helgason.
\newblock {\em Differential geometry, {L}ie groups, and symmetric spaces},
  volume~34 of {\em Graduate Studies in Mathematics}.
\newblock American Mathematical Society, Providence, RI, 2001.
\newblock Corrected reprint of the 1978 original.

\bibitem{Hump72}
James~E. Humphreys.
\newblock {\em Introduction to {L}ie algebras and representation theory}.
\newblock Springer-Verlag, New York, 1972.
\newblock Graduate Texts in Mathematics, Vol. 9.

\bibitem{Ker63}
Roy~P. Kerr.
\newblock Gravitational field of a spinning mass as an example of algebraically
  special metrics.
\newblock {\em Phys. Rev. Lett.}, 11:237--238, 1963.

\bibitem{Kna86}
Anthony~W. Knapp.
\newblock {\em Representation theory of semisimple groups}, volume~36 of {\em
  Princeton Mathematical Series}.
\newblock Princeton University Press, Princeton, NJ, 1986.
\newblock An overview based on examples.

\bibitem{Lax68}
Peter~D. Lax.
\newblock Integrals of nonlinear equations of evolution and solitary waves.
\newblock {\em Comm. Pure Appl. Math.}, 21:467--490, 1968.

\bibitem{MaMi67}
Richard~A. Matzner and Charles~W. Misner.
\newblock Gravitational field equations for sources with axial symmetry and
  angular momentum.
\newblock {\em Phys. Rev.}, 154:1229--1232, Feb 1967.

\bibitem{Maz84}
P.~O. Mazur.
\newblock A relationship between the electrovacuum {E}rnst equations and
  nonlinear {$\sigma $}-model.
\newblock {\em Acta Phys. Polon. B}, 14(4):219--234, 1983.

\bibitem{Misner}
Charles~W. Misner.
\newblock {Harmonic Maps as Models for Physical Theories}.
\newblock {\em Phys.Rev.}, D18:4510--4524, 1978.

\bibitem{NeuKra83}
G.~Neugebauer and D.~Kramer.
\newblock Einstein-{M}axwell solitons.
\newblock {\em J. Phys. A}, 16(9):1927--1936, 1983.

\bibitem{NCCEPT65}
E.~T. Newman, E.~Couch, K.~Chinnapared, A.~Exton, A.~Prakash, and R.~Torrence.
\newblock Metric of a rotating, charged mass.
\newblock {\em Jour. Math. Phys.}, 6(6):918--919, 1965.

\bibitem{Noe18}
Emmy Noether.
\newblock Invariant variation problems.
\newblock {\em Transport Theory Statist. Phys.}, 1(3):186--207, 1971.
\newblock Translated from the German (Nachr. Akad. Wiss. G{\"o}ttingen
  Math.-Phys. Kl. II 1918, 235--257).

\bibitem{Poh76}
K.~Pohlmeyer.
\newblock Integrable {H}amiltonian systems and interactions through quadratic
  constraints.
\newblock {\em Comm. Math. Phys.}, 46(3):207--221, 1976.

\bibitem{Sahi}
Siddhartha Sahi.
\newblock Rutgers, The State University of New Jersey. Private communication,
  2012.

\bibitem{NSanch82}
Norma Sanchez.
\newblock Connection between the nonlinear sigma model and the einstein
  equations of general relativity.
\newblock {\em Phys. Rev. D}, 26:2589--2597, Nov 1982.

\bibitem{ShatStrauss96}
Jalal Shatah and Walter Strauss.
\newblock Breathers as homoclinic geometric wave maps.
\newblock {\em Phys. D}, 99(2-3):113--133, 1996.

\bibitem{Stephani}
Hans Stephani, Dietrich Kramer, Malcolm MacCallum, Cornelius Hoenselaers, and
  Eduard Herlt.
\newblock {\em Exact solutions of {E}instein's field equations}.
\newblock Cambridge Monographs on Mathematical Physics. Cambridge University
  Press, Cambridge, second edition, 2003.

\bibitem{Tau49}
Olga Taussky.
\newblock A recurring theorem on determinants.
\newblock {\em The American Mathematical Monthly}, 56(10):672--676, 1949.

\bibitem{Terng10}
Chuu-Lian Terng.
\newblock Soliton hierarchies constructed from involutions.
\newblock In {\em Fourth {I}nternational {C}ongress of {C}hinese
  {M}athematicians}, volume~48 of {\em AMS/IP Stud. Adv. Math.}, pages
  367--381. Amer. Math. Soc., Providence, RI, 2010.

\bibitem{TerUhl04}
Chuu-Lian Terng and Karen Uhlenbeck.
\newblock {$1+1$} wave maps into symmetric spaces.
\newblock {\em Comm. Anal. Geom.}, 12(1-2):345--388, 2004.

\bibitem{Uhl89}
Karen Uhlenbeck.
\newblock Harmonic maps into {L}ie groups: classical solutions of the chiral
  model.
\newblock {\em J. Differential Geom.}, 30(1):1--50, 1989.

\bibitem{WehrBarut94}
R.~F. Wehrhahn and A.~O. Barut.
\newblock Symmetry scattering for {${\rm SU}(2,2)$} with applications.
\newblock {\em J. Math. Phys.}, 35(6):2838--2855, 1994.

\bibitem{Wei92}
Gilbert Weinstein.
\newblock The stationary axisymmetric two-body problem in general relativity.
\newblock {\em Comm. Pure Appl. Math.}, 45(9):1183--1203, 1992.

\bibitem{Wei96}
Gilbert Weinstein.
\newblock {$N$}-black hole stationary and axially symmetric solutions of the
  {E}instein/{M}axwell equations.
\newblock {\em Comm. Partial Differential Equations}, 21(9-10):1389--1430,
  1996.

\bibitem{Woo94}
J.~C. Wood.
\newblock Harmonic maps into symmetric spaces and integrable systems.
\newblock In {\em Harmonic maps and integrable systems}, Aspects Math., E23,
  pages 29--55. Vieweg, Braunschweig, 1994.

\bibitem{Xan84}
Basilis~C. Xanthopoulos.
\newblock A geometric notion of complete integrability.
\newblock {\em Phys. Lett. A}, 105(7):334--338, 1984.

\bibitem{ZakSha71}
V.~E. Zaharov and A.~B. {\v{S}}abat.
\newblock Exact theory of two-dimensional self-focusing and one-dimensional
  self-modulation of waves in nonlinear media.
\newblock {\em \v Z. \`Eksper. Teoret. Fiz.}, 61(1):118--134, 1971.

\bibitem{ZS-I}
V.~E. Zaharov and A.~B. {\v{S}}abat.
\newblock A scheme for integrating the nonlinear equations of mathematical
  physics by the method of the inverse scattering problem. {I}.
\newblock {\em Funktsional. Anal. i Prilozhen.}, 8(3):43--53, 1974.

\bibitem{ZS-II}
V.~E. Zaharov and A.~B. {\v{S}}abat.
\newblock Integration of the nonlinear equations of mathematical physics by the
  method of the inverse scattering problem. {II}.
\newblock {\em Funktsional. Anal. i Prilozhen.}, 13(3):13--22, 1979.

\bibitem{SG}
V.~E. Zaharov, L.~A. Tahtad{\v{z}}jan, and L.~D. Faddeev.
\newblock A complete description of the solutions of the ``sine-{G}ordon''\
  equation.
\newblock {\em Dokl. Akad. Nauk SSSR}, 219:1334--1337, 1974.

\bibitem{ZM}
V.~E. Zakharov and A.~V. Mikha{\u\i}lov.
\newblock Relativistically invariant two-dimensional models of field theory
  which are integrable by means of the inverse scattering problem method.
\newblock {\em Soviet Physics JETP}, 47:1017--1027, 1978.

\end{thebibliography}

\end{document}